\documentclass{amsart}
\usepackage{amssymb, graphics, color, enumitem, mathrsfs}
\usepackage[all]{xy}
\usepackage{appendix}

 \usepackage[applemac]{inputenc}
\usepackage[cyr]{aeguill}

\usepackage[margin = 1in]{geometry}

\newtheorem{lemma}{Lemma}[section]
\newtheorem{proposition}[lemma]{Proposition}
\newtheorem{corollary}[lemma]{Corollary}
\newtheorem{theorem}[lemma]{Theorem}

\newtheorem{example}[lemma]{Example}
\newtheorem{definition}[lemma]{Definition}
\newtheorem{remark}[lemma]{Remark}

\newtheorem*{Acknowledgement}{Acknowledgements}

\newcommand\ie{i\@.e\@. }

\newcommand\pa{ \partial}

\newcommand\bbC{\mathbb C}

\newcommand\bbH{\mathbb H}

\newcommand\bbR{\mathbb R}
\newcommand\bbS{\mathbb S}

\newcommand\bbZ{\mathbb Z}

\renewcommand\Im{\operatorname{Im}}

\usepackage{color}

\newcommand{\lrp}[1]{\left( {#1} \right)}

\newcommand\CI{\mathcal{C}^{\infty}}

\newcommand\Fr{\operatorname{Fr}}

\newcommand\cC{\mathcal{C}}

\newcommand\cD{\mathcal{D}}
\newcommand\cF{\mathcal{F}}
\newcommand\cL{\mathcal{L}}
\newcommand\db{\overline{\pa}}

\newcommand\cV{\mathcal{V}}
\newcommand\cU{\mathcal{U}}
\newcommand\cI{\mathcal{I}}

\newcommand\pr{\operatorname{pr}}

\newcommand\Id{\operatorname{Id}}

\newcommand\cH{\mathcal{H}}

\newcommand\hM{\widehat{M}}

\newcommand\cM{\mathcal{M}}

\newcommand\Tr{\operatorname{Tr}}

\newcommand\rel{\operatorname{rel}}
\newcommand\abs{\operatorname{abs}}

\newcommand\AC{\operatorname{AC}}

\newcommand\SU{\operatorname{SU}}

\newcommand\Sp{\operatorname{Sp}}
\newcommand\tr{\operatorname{tr}}

\newcommand{\Hilb}{\operatorname{Hilb}}

\newcommand{\inc}{\operatorname{in}}
\newcommand{\out}{\operatorname{out}}

\newcommand\QAC{\operatorname{QAC}}

\newcommand\QFB{\operatorname{QFB}}

\newcommand\ALE{\operatorname{ALE}}

\newcommand\QALE{\operatorname{QALE}}
\newcommand\SO{\operatorname{SO}}
\newcommand\ADHM{\operatorname{ADHM}}

\newcommand\Hom{\operatorname{Hom}}

\newcommand\cS{\mathcal{S}}

\newcommand\mf{\mathfrak}

\newcommand\cW{\mathcal{W}}

\newcommand\gU{\operatorname{U}}
\newcommand\reduced{\operatorname{red}}
\newcommand\cO{\mathcal{O}}

\begin{document}
\title[Asymptotic geometry at infinity of quiver varieties]
{Asymptotic geometry at infinity of quiver varieties}

\author{Panagiotis Dimakis}
\address{Department of Mathematics, University of Maryland}
\email{pdimakis12345@gmail.com}

\author{Fr\'ed\'eric Rochon}
\address{D\'epartement de Math\'ematiques, Universit\'e du Qu\'ebec \`a Montr\'eal}
\email{rochon.frederic@uqam.ca}

\maketitle

\begin{abstract}
Using an approach developed by Melrose to study the geometry at infinity of the Nakajima metric on the reduced Hilbert scheme of points on $\bbC^2$, we show that the Nakajima metric on a quiver variety is quasi-asymptotically conical ($\QAC$) whenever its defining parameters satisfy an appropriate genericity assumption.  As such, it is of bounded geometry and of maximal volume growth.  Being $\QAC$ is one of two main ingredients allowing us to use the work of Kottke and the second author to compute its reduced $L^2$-cohomology and prove the Vafa-Witten conjecture.  The other is a vanishing theorem in $L^2$-cohomology for exact wedge $3$-Sasakian metrics generalizing a result of Galicki and Salamon for closed $3$-Sasakian manifolds.    
\end{abstract}

\tableofcontents

\numberwithin{equation}{section}

\section{Introduction}

In \cite{Sen} Sen, based on predictions coming from a particular type of duality in string theory called $S$-duality, conjectured that the Hodge cohomology $\mathcal H^q(\widetilde{\mathcal M}_k^0)$ of the $L^2$-metric of the universal cover $\widetilde{\mathcal M}_k^0$ of the reduced moduli space $\mathcal M_k^0$ of $\SU(2)$ monopoles of magnetic charge $k$ on $\bbR^3$ is only non-trivial in middle degree and admits a complete description in terms of a natural $\bbZ_k$-action. Soon after Sen formulated his conjecture, Segal and Selby \cite{Segal-Selby} computed the relative and absolute cohomologies $H_c^q(\widetilde{\mathcal M}_k^0)$ and $H^q(\widetilde{\mathcal M}_k^0)$ and showed that the images $\Im[H_c^q(\widetilde{\mathcal M}_k^0)\to H^q(\widetilde{\mathcal M}_k^0)]$ satisfy the predictions of Sen's conjecture. They used this observation to reformulate Sen's conjecture into the statement that the natural inclusion 
\begin{equation}\label{Sen conjecture}
\Im[H_c^q(\widetilde{\mathcal M}_k^0)\to H^q(\widetilde{\mathcal M}_k^0)]\hookrightarrow \mathcal H^q(\widetilde{\mathcal M}_k^0)
\end{equation}
is in fact an isomorphism. 

Around the same time that Sen formulated his conjecture, Nakajima in \cite{Nakajima1994} generalized the $\ADHM$ construction of instantons on asymptotically locally Euclidean $(\ALE)$ spaces of \cite{Kronheimer-Nakajima} by allowing the underlying quiver and the dimensions of the vector spaces associated to the vertices of the quiver to be arbitrary. The new family of varieties $\mathfrak M _{\zeta}$ thus defined, called Nakajima quiver varieties, carry a natural metric and under the assumption that the $\zeta$ parameter is generic, are complete hyperK\"ahler manifolds. 

Shortly after \cite{Sen, Nakajima1994} appeared, Vafa and Witten in \cite{Vafa-Witten},  again based on predictions of $S$-duality, made a similar conjecture to the one of Sen about the Hodge cohomology  $\mathcal H^{q}(\mathfrak M_{\zeta})$ of a quiver variety $\mathfrak{M}_{\zeta}$. Specifically, assuming that the $\zeta$ parameter is generic, they conjectured that the middle dimensional absolute cohomology of $\mathfrak{M}_{\zeta}$ should coincide with all of the Hodge cohomology $\mathcal H^{*}(\mathfrak M_{\zeta})$ of the Nakajima metric on $\mathfrak M_{\zeta}$. It follows from \cite[Corollary 11.2]{Nakajima_1998} that the natural map $H_c^q(\mathfrak M_{\zeta})\to H^q(\mathfrak M_{\zeta})$ from compactly supported into absolute cohomology is an isomorphism in middle degree. Therefore, we can restate the conjecture in the form \eqref{Sen conjecture} by asking that the natural inclusion 
\begin{equation}
\Im[H_c^q(\mathfrak M_{\zeta})\to H^q(\mathfrak M_{\zeta})]\hookrightarrow \mathcal H^q(\mathfrak M_{\zeta})
\end{equation}
be in fact an isomorphism. 

A major step towards a proof of both conjectures was made by Hitchin \cite{Hitchin}, who showed that for many hyperK\"ahler metrics coming from hyperK\"ahler quotient constructions, all $L^2$-harmonic forms lie in middle dimension, thus immediately proving both conjectures outside of middle degree. The Vafa-Witten conjecture for the simplest type of Nakajima quiver variety, namely the $\ALE$ gravitational instantons,  follows from standard results on the $L^2$-cohomology of asymptotically conical $(\AC)$ metrics, see for example \cite{MelroseGST,HHM2004}. More recently, attention was focused on the particular case of the Vafa-Witten conjecture when the Nakajima quiver variety is the reduced Hilbert scheme of $n$ points on $\bbC^2$, $\Hilb_0^n(\bbC^2)$. The $n=3$ case of the conjecture was proven by Carron in \cite{Carron2011b} and in \cite{Carron2011} it was proven that the Nakajima metric on $\Hilb_0^n(\bbC^2)$ is quasi-asymptotically locally Euclidean $(\QALE)$ in the sense of Joyce \cite{Joyce}. The case for arbitrary $n$ was settled  by Kottke and the second author in \cite{KR2} using the analytical results of \cite{KR1}.

In this paper we study the asymptotic geometry of Nakajima quiver varieties and using the strategy of \cite{KR2} and the results of \cite{KR1}, we prove the Vafa-Witten conjecture for Nakajima quiver varieties $\mathfrak M_{\zeta}$ with $\zeta$ satisfying a slightly stronger genericity assumption than the one required to guarantee smoothness of $\mathfrak M_{\zeta}$.
\begin{theorem}
The Vafa-Witten conjecture holds for all Nakajima quiver varieties $\mathfrak M_{\zeta}$ under the assumption that $\zeta$ is properly generic in the sense of Definition~\ref{pg.2} below.\end{theorem}
In order to apply the results of \cite{KR1,KR2}, we need to show that the Nakajima metric on $\mathfrak M_{\zeta}$ is a quasi-fibered boundary $(\QFB)$ metric in the sense of \cite{CDR}. In fact, we show that they are quasi-asymptotically conical ($\QAC$), a particular type of $\QFB$ metrics of maximal volume growth originally introduced by Degeratu and Mazzeo \cite{DM2018} and generalizing the notion of $\QALE$ metrics.

\begin{theorem}
Given that $\zeta$ is properly generic, the Nakajima metric of any Nakajima quiver variety $\mathfrak M_{\zeta}$ is $\QAC$ and admits a smooth expansion at infinity in the sense of Definition~\ref{QFBdef.1} below.  In particular, it is of bounded geometry and of maximal volume growth.        
\end{theorem}

Our strategy to prove this result is strongly inspired by an approach developed by Melrose \cite{Melrose_CIRM, Melrose_London} to give a geometric proof of the result of Carron \cite{Carron2011} and show that the metric has a smooth expansion at infinity. 
 We start by radially compactifying the Nakajima quiver representation space $\bold M$. The action of the group $G$ of gauge transformations on $\bold M$ is unitary and thus extends to the radial compactification $\overline{\bold M}$. Using the result of \cite{AM2011}, we resolve the group action at the boundary by iteratively blowing up the boundary strata indexed by conjugacy classes of stabilizer subgroups of $G$. The resulting space $\widetilde{\bold M}$ is called the $\QAC$ compactification of $\bold M$. A careful analysis of the hyperK\"ahler moment map $\mu$ shows that the closure $\widetilde{\mu^{-1}(-\zeta)}$ of $\mu^{-1}(-\zeta)$ into $\widetilde{\bold M}$ is naturally a manifold with fibered corners with induced metric a $\QAC$ metric. Since the whole construction is $G$-equivariant and $G$ acts freely on $\widetilde{\mu^{-1}(-\zeta)}$, the metric descends to a $\QAC$ metric on $\mathfrak M_{\zeta}$. 
\begin{remark}
It was proven in \cite{Weiss2021} that there exist quiver varieties whose associated Nakajima metric is not $\QALE$ and it was asked whether the Nakajima metric on those varieties is $\QAC$. Our result gives a positive answer to their question. 
\end{remark}

Our analysis of the behaviour of the moment map near the boundary further implies that if $H$ is a boundary hypersurface of  $\widetilde{\bold M}$ and $\phi_H: H\cap \widetilde{\mu^{-1}(-\zeta)}  \to \Sigma_H$ is the fiber bundle of the corresponding boundary hypersurface of $\widetilde{\mu^{-1}(-\zeta)}$, then each fiber is the $\QAC$ compactification of a quiver variety of lower dimension. Our assumption that $\zeta$ is properly generic guarantees that the Nakajima quiver varieties appearing as the interiors of these fibers are smooth. 

The base $\Sigma_H$ turns out to be an incomplete $3$-Sasakian manifold and the induced metric $g_{S_H}$ is an exact wedge $3$-Sasakian metric. In order to be able to apply the results of \cite{KR1}, we need to show that the Hodge-deRham operator associated to some flat Euclidean vector bundle $E\to S_H$ has no $L^2$-cohomology in certain degrees. Specifically, referring to Theorem~\ref{bs.22} for further details, we prove 

\begin{theorem}
Let $\cS$ be a manifold with fibered corners of dimension $4n+3$ and $E\to \cS$ a nicely $\Sp(1)$-equivariant flat Euclidean vector bundle in the sense of Definition~\ref{inv.3}.  Suppose that $g_w$ is an exact wedge $3$-Sasakian metric on $\cS$ and let $\eth_w$ be the Hodge-deRham operator associated to $g_w$ and $E$.  Then for $k\in \{0,\ldots,n\}$, the $L^2$-kernel of $\eth_w$ is trivial when $\eth_w$ is acting on forms of degree $2k+1$. 
In particular, when $\cS$ is a closed manifold, this implies the vanishing theorem of Galicki-Salamon \cite{Galicki-Salamon1996}, namely that the space of harmonic forms in degree $2k+1$ is trivial for $k\in \{0,\ldots,n\}$.
\label{int.1}\end{theorem}

To prove this result, we follow the overall strategy of \cite{Galicki-Salamon1996}, which relies on a result of Tachibana \cite{Tachibana} concerning harmonic forms on a closed Sasakian manifold.  The original proof of Tachibana involves integrations by parts that seem difficult to justify in our singular setting.  We instead proceed differently with a new proof of Tachibana's results relying on transverse Hodge theory and the transverse Hard Lefschetz theorem \cite{EKA1986,EKA1990}, see in particular Remarks~\ref{van.5} and \ref{van.7b} below.  This new proof can be adapted to our singular setting by obtaining suitable versions of Hodge decomposition and the Hard Lefschetz theorem for the $L^2$-cohomology of a wedge K\"ahler  metric, namely Proposition~\ref{l2i.19} and Corollary~\ref{l2i.23} below.

The paper is organized as follows. In \S~\ref{qv.0} we review basic facts about Nakajima quiver varieties. In  \S~\ref{mwfc.0} we introduce manifolds with fibered corners and define $\QFB$ and wedge metrics. In  \S~\ref{cqac.0} we construct the $\QAC$ compactification of the Nakajima quiver representation space and use it to show that the natural hyperK\"ahler metric on $\mathfrak M_{\zeta}$ is an exact $\QAC$ metric. In \S~\ref{l2i.0}, we recall some important results concerning the $L^2$-cohomology of wedge metrics and derive the versions of the Hodge decomposition and the Hard Lefschetz theorem that we will need when these metrics are K\"ahler.   In \S~\ref{bs.0} we introduce exact wedge $3$-Sasakian metrics and show that their $L^2$-harmonic forms are $\Sp(1)$-invariant.  This is used in \S~\ref{van.0} to establish a version of Tachibana's results \cite{Tachibana} for our singular setting and prove Theorem~\ref{int.1}. Finally, in  \S~\ref{rcqv.0} we use the results of sections \ref{cqac.0} and \ref{van.0} along with the results of \cite{KR1,KR2} to prove the Vafa-Witten conjecture. 

\begin{Acknowledgement}  
We are grateful to Vestislav Apostolov and Gilles Carron for helpful discussions, as well as to an anonymous referee for pointing out the incompleteness of an argument appearing in an earlier version of this paper.  The second author was supported by a NSERC discovery grant and a FRQNT team research project grant.
\end{Acknowledgement}

\section{Quiver varieties}\label{qv.0}

Given a finite graph $\Gamma$ with $n$ vertices, let $H$ be the set of pairs consisting of an edge together with an orientation of it. For $h\in H$, let $\inc(h)$ denote the incoming vertex, $\out(h)$ the outgoing vertex and $\bar h$ the edge with reversed orientation.  We will allow for loops, that is, for edges $h$ such that $\inc(h)=\out(h)$.  Choose a subset $\Omega \subset H$ such that $\overline\Omega\cup\Omega = H$ and $\overline\Omega\cap\Omega = \emptyset$, where $\overline{\Omega}=\{\overline{h}\; | \; h\in\Omega\}$. Such an $\Omega$ corresponds to a choice of \emph{orientation} for the graph.  


Given a pair of Hermitian vector spaces $(V_k,W_k)$ for each vertex $k$, we define the \textbf{Nakajima quiver representation space} 
\begin{equation}
\bold{M}(\bold{v},\bold{w}):= \left( \bigoplus\limits_{h\in H} \Hom(V_{\out(h)},V_{\inc(h)})\right) \oplus \left(\bigoplus\limits_{k=1}^n \Hom(W_k,V_k) \oplus \Hom(V_k,W_k)\right),
\end{equation}
where 
\begin{equation*}
\bold v := (\dim_{\bbC} V_1,...,\dim_{\bbC} V_n), ~~ \bold w := (\dim_{\bbC} W_1,...,\dim_{\bbC} W_n).
\end{equation*}

If $B_{h} \in \Hom(V_{\out(h)},V_{\inc(h)}),~ i_k \in \Hom(W_k,V_k)$ and $j_k \in \Hom(V_k,W_k)$, then the elements of $\bold{M}$ are denoted as triples $(B,i,j)$ where $B$ denotes the collection $(B_h)_{h\in H}$, $i$ the collection $(i_k)_{1\le k \le n}$ and $j$ the collection $(j_k)_{1\le k \le n}$. 

We define on $\bold{M}$ the holomorphic symplectic form 

\begin{equation}\label{symplectic form}
\omega_{\bbC}((B,i,j),(B',i',j')) := \sum\limits_{h\in H} \tr(\epsilon(h)B_hB'_{\bar h}) + \sum\limits_{k=1}^n \tr(i_kj_k' - i_k'j_k),
\end{equation}
where $\epsilon(h) =1 $ if $h \in \Omega$ and $\epsilon(h) = -1$ if $h\in\overline{\Omega}$. The symplectic vector space $\bold{M}$ decomposes into the sum $\bold{M} = \bold{M}_{\Omega}\oplus\bold{M}_{\overline{\Omega}}$ of Lagrangian subspaces: 

\begin{equation}
\begin{split}
\bold{M}_{\Omega} :=& \left(\bigoplus\limits_{h\in \Omega} \Hom(V_{\out(h)},V_{\inc(h)})\right) \oplus \left(\bigoplus\limits_{k=1}^n \Hom(W_k,V_k)\right),\\
\bold{M}_{\overline\Omega} :=& \left(\bigoplus\limits_{h\in \overline\Omega} \Hom(V_{\out(h)},V_{\inc(h)})\right) \oplus \left(\bigoplus\limits_{k=1}^n \Hom(V_k,W_k)\right).
\end{split}
\end{equation}

Given this decomposition, we introduce a new complex structure $J$ given by $J\cdot (m,m') = (-m'^{\dagger}, m^{\dagger})$ for $(m,m')\in \bold{M}_{\Omega}\oplus \bold{M}_{\overline\Omega}$. The complex structures $I,J$ together endow $\bold{M}$ with a hyperK\"ahler structure. Given $\omega_{\bbC}$ and $J$ we define the metric 

\begin{equation}
g((B,i,j),(B',i',j')) := \omega_{\bbC}((B,i,j),J\cdot(B',i',j')) = \sum\limits_{h\in H} \tr(B_hB_h^{'\dagger}) + \sum\limits_{k=1}^n \tr(i_k i_k^{\dagger} + j_k^{'\dagger}j_k).
\end{equation}

The group $G = \prod \gU(V_k)$ acts on $\bold{M}$. The action, which is given by 

\begin{equation*}
g\cdot(B,i,j) = (g_{\inc(h)}B_hg_{\out(h)}^{-1}, g_k i_k,j_k g_k^{-1}),
\end{equation*}
preserves the metric and the hyperK\"ahler structure. Let $\mu$ be the corresponding hyperK\"ahler moment map. Writing $\mu = (\mu_{\bbR},\mu_{\bbC})$, we have that 

\begin{equation}
\begin{split}
\mu_{\bbR}(B,i,j) &= \frac{\sqrt{-1}}{2}\left( \sum\limits_{h\in H: k = \inc(h)} B_hB_h^{\dagger} - B_{\bar{h}}^{\dagger}B_{\bar{h}} +i_ki_k^{\dagger} - j_k^{\dagger}j_k\right)_k \in \bigoplus\limits_k \mf{u}(V_k) = \mf{g},\\
\mu_{\bbC}(B,i,j) &= \left(\sum\limits_{h\in H: k = \inc(h)} \epsilon(h)B_hB_{\bar{h}} + i_k j_k\right)_k \in \mf{gl}(V_k) = \mf{g}\otimes\bbC.
\end{split}
\end{equation}
Here, we identify $\mathfrak g$ with its dual $\mathfrak g^{*}$. When the graph $\Gamma$ has loops, we also consider the reduced vector space 
$$
\bold{M}_{\reduced}:= \{ (B,i,j)\in \bold{M} ~|~ \Tr(B_h)=0 \; \mbox{whenever $h$ is a loop}\}\subset \bold{M}.  
$$ 
Notice that the induced metric on $\bold{M}_{\reduced}$ is also hyperK\"ahler and that $\bold{M}_{\reduced}$ is invariant under the action of $G$.  

\begin{definition}\label{Quiver variety}
Given an element $\zeta = (\zeta_{\bbR},\zeta_{\bbC}) \in Z\oplus ( Z\otimes\bbC)$, where $Z\subset \mf{g}$ denotes the center, consider the hyperK\"ahler quotient 

\begin{equation}
\mathfrak M_{\zeta} := \{(B,i,j) \in \bold{M}~|~\mu(B,i,j) = -\zeta\}/G.
\end{equation}

This is the \emph{quiver variety} associated to $H$ and $\zeta$. If $H$ contains loops, we also consider the reduced quiver variety 
\begin{equation}
\mathfrak{M}_{\zeta}^{\reduced} := \{(B,i,j) \in \bold{M}_{\reduced}~|~\mu(B,i,j) = -\zeta\}/G.
\end{equation}
\end{definition}

We want to know when such a quiver variety is smooth. Let $A$ be the adjacency matrix of the graph $\Gamma$, meaning 

$$
A_{i,j} = A_{j,i} = \textrm{number of edges joining vertex }i\textrm{ to vertex }j
$$
and $C = 2\Id - A$ the generalized Cartan matrix. Since the center $Z$ of $\mf{g}$ is the product of scalar matrices on $V_k$, we may identify $Z$ with a subspace of $\bbR^n$. 
Let us introduce the following notation: 

\begin{equation}
\begin{split}
R_+ :=& \{\theta = (\theta_k)\in \bbZ^n_{\ge 0} | ^t\theta C \theta \le 2\} \backslash \{0\}, \\
R_+(\bold v) :=& \{\theta\in R_+ | \theta_k \le \dim_{\bbC} V_k~\forall~k\},\\
D_{\theta}:=& \left\{ x = (x_k) \in \bbR^n \Big| \sum\limits_k x_k\theta_k =0 \right\}\textrm{ for } \theta\in R_+,\\
Z_s(\mathfrak g) :=& \bigcup\limits_{\theta\in R_+(\bold{v})} \bbR^3\otimes D_{\theta}.
\end{split}
\label{gen.1}
\end{equation}
\begin{theorem}[\cite{Nakajima1994} Theorem 2.8 and Corollary 4.2]\label{generic smoothness}
Suppose that

\begin{equation}\label{generic}
\zeta \in (\bbR^3\otimes\bbR^n) \backslash Z_s(\mathfrak g).
\end{equation}

Then the quiver variety $\mathfrak{M}_{\zeta}$ is nonsingular and the induced hyperK\"ahler metric is complete. We say parameters $\zeta$ satisfying \eqref{generic} are \textbf{generic}. Given generic parameters $\zeta \neq \zeta'$, $\mathfrak M_{\zeta}$ and $\mathfrak M_{\zeta'}$ are diffeomorphic. 
\end{theorem}

\section{Manifolds with fibered corners} \label{mwfc.0}

Let $M$ be a compact manifold with corners in the sense of \cite{Grieser, MelroseMWC, Melrose1992} and  $\mathcal M_1(M)$ be its set of codimension $1$ corners (also called boundary hypersurfaces). 
\begin{definition}\label{MWFC}
For each $H\in\cM_1(M)$, let $\phi_H:H\to S_H$ be a fiber bundle over a compact manifold with corners $S_H$ and denote by $\phi$ the collection of these fiber bundles. We say that $\phi$ is an \textbf{iterated fibration structure} for $M$ if there exists a partial ordering on the boundary hypersurfaces $H$ such that:
\begin{itemize}
\item Any subset $\mathcal I \subset \cM_1(M)$ of boundary hypersurfaces such that $\bigcap\limits_{H\in \mathcal I} H \neq \emptyset$ is totally ordered;
\item If $H_1,H_2\in \cM_1(M)$ are such that $H_1 <H_2$, then $H_1\cap H_2 \neq \emptyset$, $\phi_{H_1}|_{H_1\cap H_2}: H_1\cap H_2 \to S_{H_1}$ is a surjective submersion, $S_{H_2H_1}:= \phi_{H_2}(H_1\cap H_2)$ is a boundary hypersurface of the manifold with corners $S_{H_2}$ and there is a surjective submersion $\phi_{H_2H_1} : S_{H_2H_1} \to S_{H_1}$ such that $\phi_{H_2H_1}\circ\phi_{H_2} = \phi_{H_1}$ on $H_1\cap H_2$;
\item The boundary hypersurfaces of $S_{H}$ are given by $S_{HH'}$ for $H' < H$. 
\end{itemize}
In this case, we say that the pair $(M,\phi)$ is a \textbf{manifold with fibered corners}. 
\end{definition}

For each boundary hypersurface $H$, let $x_{H}\in C^{\infty}(M)$ be a boundary defining function, that is, $x_{H}$ takes nonnegative values, $x_{H}^{-1}(0) = H$ and $dx_{H}$ is nowhere zero along $H$. The boundary defining functions $x_{H}$ are \textbf{compatible} with the iterated fibration structure $\phi$ if $x_{H}$ restricted to $H'$ is constant along the fibers of $\phi_{H'}:  H'\to S_{H'}$ whenever $H' > H$. 
\begin{definition}\label{QFBvf}
Let $v = \prod\limits_{H\in \cM_1(M)} x_{H}$ be a total boundary defining function for the manifold with fibered corners $(M,\phi)$. For such a choice, the space $\mathcal V_{\QFB}(M)$ of \textbf{quasi-fibered boundary} $(\QFB)$ vector fields consists of smooth vector fields $\xi$ in $M$ such that:
\begin{enumerate}
\item  $\xi$ is tangent to the fibers of $\phi_{H}:H \to S_{H}$ for each boundary hypersurface $H$ of $M$;
 \item $\xi v \in v^2C^{\infty}(M)$.
\end{enumerate}
\end{definition}
\begin{remark}
The condition $\xi v \in v^2C^{\infty}(M)$ clearly depends on the choice of the total boundary defining function $v$, but by \cite[Lemma~1.1]{KR1}, two total boundary defining functions $v$ and $v'$ give the same space of $\QFB$ vector fields if and only if the function $\frac{v}{v'}$ is constant on the fibers of $\phi_{H}$ for all $H$.  When this is the case, we say that the total boundary defining functions $v$ and $v'$ are \textbf{$\QFB$ equivalent}. 
\end{remark}

As explained in \cite{CDR}, there exists a natural bundle $^{\QFB}TM\to M$ called the $\QFB$ tangent bundle and a natural map
\begin{equation}
a_{\QFB}: ^{\QFB}TM\to TM
\end{equation}
inducing a canonical isomorphism
\begin{equation}
(a_{\QFB})_{*}:C^{\infty}(M; ^{\QFB}TM) \to \mathcal V_{\QFB}(M).
\end{equation}
This gives ${}^{\QFB}TM\to M$ a structure of Lie algebroid with anchor map $a_{\QFB}$.   

\begin{definition}
A \textbf{quasi-fibered boundary} $(\QFB)$ metric on a manifold with fibered corners $(M,\phi)$ equipped with a choice of total boundary defining function $v$ is a Riemannian metric on the interior of $M$ of the form 
\begin{equation}
(a_{\QFB})_{*}(h|_{M\backslash \pa M})
\end{equation}
for some choice of bundle metric $h\in C^{\infty}(M; S^2(^{\QFB}T^{*}M))$ for the vector bundle $^{\QFB}TM$. In this case we say that the manifold with corners $M$ is the $\QFB$ \textbf{compactification} of the corresponding Riemannian manifold. 
\label{QFBdef.1}\end{definition}
A $\QFB$ metric on $(M,\phi)$ such that for each maximal hypersurface $H$, $S_{H} = H$ and $\phi_{H} = \Id$ will be called a \textbf{quasi-asymptotically conical} $(\QAC)$ metric. 
For the purposes of the paper, we introduce two more types of metrics on manifolds with fibered corners. If we drop condition $(2)$ from definition \ref{QFBvf} we obtain the Lie algebra of edge vector fields $\mathcal V_e(M)$ and the edge tangent bundle $^e TM\to M$, a naturally associated Lie algebroid with anchor map $a_e: {}^{e}TM\to TM$. As with $\mathcal V_{\QFB}(M)$, to the Lie algebra $\mathcal V_e(M)$ one can associate the class of edge metrics on $M\backslash \pa M$. Both $\QFB$ metrics $(M\backslash \pa M, g_{\QFB})$ and edge metrics $(M\backslash \pa M, g_e)$ are examples of Riemannian manifolds with Lie structure at infinity in the sense of \cite{ALN04}. By \cite{ALN04} and \cite{Bui}, such metrics are complete metrics of infinite volume with bounded geometry.  Moreover, by \cite{ALN04}, two such metrics $g$ and $g'$ corresponding to a fixed Lie structure at infinity are quasi-isometric, meaning that there exists a constant $C>0$ such that
\begin{equation*}
\frac{g}{C} < g' <Cg.
\end{equation*}

In order to introduce \textbf{wedge} metrics, we define the wedge cotangent bundle 
\begin{equation}
^w T^{*}M := v(^e T^{*}M)
\end{equation}
with $v$ a total boundary defining function. 

\begin{definition}
A \textbf{wedge} metric (also called incomplete iterated edge metrics in \cite{ALMP2012}) on a manifold with fibered corners $(M,\phi)$ is a Riemannian metric $g_w$ on the interior of $M$ of the form 
\begin{equation}
g_w = v^2g_e
\end{equation}
for some edge metric $g_e$.
\end{definition}

A wedge metric $g_w$ is of finite volume and is geodesically incomplete, so the pair $(M \backslash \pa M, g_w)$ is not a Riemannian manifold with Lie structure at infinity. This is evident from the fact that the wedge vector fields $\xi\in \CI(M;{^wTM})$ are not preserved by the Lie bracket.  When $M$ is compact, the metric completion of $(M\setminus\pa M, g_w)$ is  the singular space $\widehat{M}_\phi$ obtained from the manifold with fibered corners $(M,\phi)$ by collapsing the fibers of $\phi_H$ onto their base for each boundary hypersurface $H\in \cM_1(M)$.  In particular, the metric completion is a \textbf{smoothly stratified space} in the sense of \cite{ALMP2012}, namely a singular space of the form $\widehat{M}_\phi$ for some manifold with fibered corners $(M,\phi)$ (the \textbf{resolution} of the smoothly stratified space $\widehat{M}_\phi$).  If $q_\phi: M\to \widehat{M}_\phi$ is the natural map, the (open) strata are $M\setminus \pa M$ (the regular stratum) and $\mathfrak{s}_H:= q_\phi(\phi_H^{-1}(S_H\setminus\pa S_H))$ for $H\in \cM_1(M)$ with closure $\overline{\mathfrak{s}}_H= q_\phi(\phi_H^{-1}(S_H))$ corresponding to the smoothly stratified manifold associated to $S_H$ equipped with the iterated fibration structure induced from $\phi$.  

In this paper, we will need to use exact wedge metrics in the sense of \cite[Definition~8.4]{KR1}.  To explain what this is, we need first to recall what are wedge metrics of product type. Let $c_H: H \times [0,\delta_{H}) \to M$ be a collar neighbourhood of $H$ compatible with the boundary defining functions in the sense that $c_H^{*}x_H$ corresponds to the projection $\pr_2: H \times [0,\delta_H) \to [0,\delta_H)$ and $c_H^{*}x_{H'}$ is the pullback of a function on $H$ for $H'\ne H$ with $H'\cap H\ne \emptyset$. Choose a connection on the fiber bundle $\phi_H : H \to S_H$.  Let $\kappa_H$ be a family of fiberwise edge metrics in the fibers of $\phi_H:H_H \to S_H$. Using the connection on $\phi_H$, this family can be lifted to a vertical symmetric $2$-tensor on $H\backslash \pa H$. Let $\pr_1:H \times [0,\delta_{H}) \to H$ be the projection onto the first factor. A \textbf{product type} wedge metric near $H$ is given by 
\begin{equation}
g_w = \rho_H^2(dx_H^2 + \pr_1^{*}\phi_H^{*}g_{S_H} + x_H^2 \pr_1^{*}\kappa_{w,H}), \quad with \quad \rho_H=\prod_{H'<H} x_{H'},
\end{equation}
where $g_{S_H}$ is an edge metric on $S_H$ and $\kappa_{w,H}=\left(\frac{v_H^2}{x_H^2}\right)\kappa_H$ with $v_H=\prod_{H'\ge H} x_{H'}$ is a $2$-tensor inducing a wedge metric on the fibers of $\phi_H:H_I\to S_H$ in such a way that $\phi_H^*g_{S_H}+\kappa_{w,H}$ is a Riemannian metric turning $\phi_H$ into a Riemannian submersion onto $(S_H,g_{S_H})$.   Notice in particular that $g_{S_H,w}:= \rho_H^2g_{S_H}$ is the natural wedge metric induced by $g_w$ on $S_H$.  
\begin{definition}
An \textbf{exact} wedge metric is a wedge metric which is of product type near $H$ up to a term in $x_HC^{\infty}(M; S^2(^wT^{*}M)) $ for each boundary hypersurface $H$ of $M$. 
\end{definition}

There is also a notion of product type and exact $\QFB$ metrics. The definition of a product type $\QFB$ metric is very similar to the one of a product wedge metric, but this time, instead of the level sets of $x_H$, we use the level sets of some total boundary defining function $\prod_H x_H$ $\QFB$ equivalent to the total boundary defining function $v$ defining $\cV_{\QFB}(M)$.  To indicate that we could take different total boundary defining functions compatible with the Lie algebra of $\QFB$ vector fields for different boundary hypersurfaces, we will denote this total boundary defining function by $u_H$.  With this understood, we consider the open set 
\begin{equation}
\mathcal U_H = \{ (p,\tau) \in H\times [0,\delta_H) \quad | \quad \prod\limits_{H'\neq H} x_{H'}(p) > \frac{\tau}{\delta_H} \} \subset H\times [0,\delta_H)
\end{equation}
with natural diffeomorphism 
\begin{equation}
\begin{array}{lccc}
\psi_H: & (H\backslash \pa H)\times [0,\delta_H) &\to &\mathcal U_H\\
&(p, t) &\mapsto & (p,t\prod\limits_{H'\neq H }x_{H'}(p)).
\end{array}
\end{equation}

Following the definition of a product type wedge metric, we choose a connection for the fiber bundle $\phi_H: H \to S_H$ and let $g_{S_H}$ be a wedge metric. Let $\kappa_H$ be a family of fiberwise $\QFB$ metrics in the fibers of $\phi_H:H \to S_H$ and use the connection to lift them to a vertical symmetric $2$-tensor on $H\backslash \pa H$. In $\mathcal U_H$  seen as a subset of $H\times [0,\delta_H)$, a $\QFB$ metric of \textbf{product type} near $H$ is a metric of the form
\begin{equation}
g_{\QFB} = \frac{\,du_H^2}{u_H^4} + \frac{\pr_1^{*}\phi_H^{*}g_{S_H}}{u_H^2} + \pr_1^{*}\kappa_H.
\label{pt.1}\end{equation}
More generally, an \textbf{exact} $\QFB$ metric is a $\QFB$ metric which is of product type near each hypersurface $H$ up to a term in $x_HC^{\infty}(M; S^2(^{\QFB}T^{*}M))$.

\section{Compactification of Nakajima quiver representation spaces} \label{cqac.0}

For $\zeta=0$, the quiver variety $\mathfrak{M}_{\zeta}$ is singular.  However, as described in \cite{Nakajima1994}, it has a natural structure of stratified space.  Let $\bold{q}: \bold{M}\to \bold{M}/G$ be the natural quotient map.  

\begin{lemma}
A point $m\in \mu^{-1}(-\zeta)$ is a singular point of $\mu^{-1}(-\zeta)$ if and only if $\bold{q}(m)\in \mathfrak{M}_{\zeta}$ is a singular point of $\mathfrak{M}_{\zeta}$.
\label{cqac.1}\end{lemma}
\begin{proof}
If $m\in \mu^{-1}(-\zeta)$ is singular, then the differential of $\mu$ at $m$ is not surjective.  By the definition of a hyperK\"ahler moment map, this means that the stabilizer of $m$ is non-trivial, so the corresponding point $x\in \mu^{-1}(-\zeta)/G$ is singular.  Conversely, if $\bold{q}(m)\in \mu^{-1}(-\zeta)/G$ is singular, then the stabilizer $G_m$ of  $m$ is non-trivial.  By \cite[p.391]{Nakajima1994}, this can only happen if the Lie algebra $\mathfrak{g}_m$ of $G_m$ is non-trivial.  By the definition of a hyperK\"ahler moment map, this means that $d\mu$ is not surjective at $m$, hence that $\mu^{-1}(-\zeta)$ is not smooth at $m$.
\end{proof}

More generally, the quotient $\bold{M}/G$ is singular, but it is a smoothly stratified space by \cite{AM2011} with strata given by
$$
     \bold{M}^I/G= \{m\in \bold{M}\; | \; G_m\in I\}/G
$$
for $I$ a conjugacy class of a stabilizer subgroup.  Strictly speaking, the result of \cite{AM2011} is for compact manifolds with corners, but since the action of $G$ naturally extends to the radial compactification $\overline{\bold{M}}$ of $\bold{M}$, it suffices to apply the result of \cite{AM2011} to $\overline{\bold{M}}$ and restrict it to $\bold{M}$.  Even if $\bold{M}$ itself is smooth, this induces a corresponding structure of smoothly stratified space with strata given by 
$$
  \bold{M}^I=  \{m\in \bold{M}\; | \; G_m\in I\}
$$
for $I$ a conjugacy class of a stabilizer subgroup.  The regular stratum is the one corresponding to the conjugacy class of the trivial stabilizer $G_m=\{\Id\}$, while the deepest stratum is the origin and corresponds to the conjugacy class of the stabilizer $G_m=G$.  Thus, if $m\in \bold{M}\setminus\{0\}$ is such that $\bold{q}(m)\in\bold{M}/G$ is singular, its stabilizer $G_m$ is non-trivial and is strictly contained in $G$ unless $m$ is contained in the orthogonal complement of $\bold{M}_{\reduced}$, in which case $G_m=G$. Let $T_m\cO_m\subset T_m\bold{M}$ be the tangent space at $m$ of the orbit $\cO_m$ of $m$.  Let $\widehat{\bold{M}}_m$ be the orthogonal complement in $T_m \bold{M}$ of the $\bbH$-module
$$
     \bbH T_m\cO_m= T_m\cO_m \oplus I_1T_m\cO_m \oplus I_2 T_m\cO_m\oplus I_3 T_m\cO_m.
$$ 
Clearly, $\widehat{\bold{M}}_m$ is itself a $\bbH$-module and the action of $G$ on $\bold{M}$ induces an action of $G_m$ on $\widehat{\bold{M}}_m$.  Using the canonical identification of $T_m\bold{M}$ with $\bold{M}$, $\widehat{\bold{M}}_m$ can be seen as an $\bbH$-submodule of $\bold{M}$.  Moreover, since the action of $G$ commutes with the natural $\bbR^+$ action on $\bold{M}$, there is a canonical identification $\widehat{\bold{M}}_{\lambda m}= \widehat{\bold{M}}_{m}$ for any $\lambda\in \bbR^+$.  
\begin{lemma} \label{moment map restriction}
Fix $m\in \mu^{-1}(0)\backslash\{0\} \subset \bold{M}\backslash\{0\}$ with non-trivial stabilizer $G_m$. The restriction $\hat{\mu}_m$ of the moment map $\mu$ to $\widehat{\bold{M}}_m$ induces a hyperK\"ahler moment map with values in $\mathfrak{g}^*_m$ for the action of $G_m$ on $\widehat{\bold{M}}_m$, where $\mathfrak{g}_m$ is the Lie algebra of $G_m$.
\label{cqac.2}
\end{lemma}  
\begin{proof}
Indeed, since 
\begin{equation}
\langle d\mu_i(v), \xi\rangle = \omega_i(\xi_*,v)=g(I_i\xi_*,v) \quad \forall \xi\in\mathfrak{g},
\label{cqac.3}
\end{equation}
where $\xi_*\in\CI(\bold{M};T\bold{M})$ is the vector field generated by $\xi$, we see that
\begin{equation}\label{zero derivative}
\langle d\mu, \xi\rangle|_{\widehat{\bold{M}}_m}=0
\end{equation}
for $\xi\in (\mathfrak{g}_m^{\perp})^*$, where $\mathfrak{g}^{\perp}_m$ is the orthogonal complement of $\mathfrak{g}_m$ in $\mathfrak{g}$.  Since $m\in \mu^{-1}(0)$, notice that $m\in \widehat{\bold M}_m$ where $\widehat{\bold M}_m$ is seen as a subspace of $\bold M$. This observation, together with the vanishing of the the derivative \eqref{zero derivative} implies that
$$
\langle \mu, \xi\rangle|_{\widehat{\bold{M}}_m}=0 \quad \forall \xi\in (\mathfrak{g}_m^{\perp})^*.  
$$
Hence, the restriction of $\mu$ to $\widehat{\bold{M}}_m$ takes values in $\mathfrak{g}_m^{*}$ and therefore corresponds to the hyperK\"ahler moment map of the action of $G_m$ on $\widehat{\bold{M}}_m$.
\end{proof}
Now, recall from \cite[Lemma~6.5]{Nakajima1994}  that since $m\in \mu^{-1}(0)\backslash\{0\} \subset \bold{M}\backslash\{0\}$ has non-trivial stabilizer $G_m$, there is an induced orthogonal decomposition

\begin{equation}
  V= V^{(0)}\oplus (V^{(1)})^{\oplus \hat{v}_1}\oplus \cdots \oplus (V^{(r)})^{\oplus \hat{v}_r}
\label{cqac.5}\end{equation}

of $V= \oplus_{k=1}^n V_k$.  On $\widehat{\bold{M}}_{m}$, this induces a decomposition
\begin{equation}
\widehat{\bold{M}}_m= (\widehat{\bold{M}}_m\cap \bold{M}(v^{(0)},w))\oplus \lrp{ \bigoplus_{i,j=1}^r \widehat{M}_{ij}\otimes \Hom(\bbC^{\hat{v}_i},\bbC^{\hat{v}_j}) } \oplus
\lrp{ \bigoplus_{i=1}^r \Hom(\bbC^{\hat{v}_i},\widehat{W}_i) }\oplus \lrp{ \bigoplus_{i=1}^r \Hom(\widehat{W}_i,\bbC^{\hat{v}_i})}, \label{cqac.6}\end{equation}

 where
 
 $$
    \widehat{M}_{ij}:=\widehat{\bold{M}}\cap \lrp{ \bigoplus_{h\in H} \Hom(V^{(i)}_{\out(h)}, V^{(i)}_{\inc(h)} )   }
 $$
and
$$
   \widehat{W}_i:=\widehat{\bold{M}}\cap \left\{   \lrp{ \bigoplus_{h\in H}\Hom(V^{(0)}_{\out(h)},V^{(i)}_{\inc(h)})} \oplus  \lrp{\bigoplus_{k=1}^n \Hom(W_k,V^{(i)}_k)} \right\}.
$$
In this decomposition, the stabilizer $G_m$ is given by 

\begin{equation}
     G_m= \prod_{i=1}^r \gU(\hat{v}_i)
\label{cqac.7}\end{equation}
and acts trivially on the component

\begin{equation}
\bold{T}_m:=  (\widehat{\bold{M}}_m\cap \bold{M}(v^{(0)},w))\oplus \lrp{ \bigoplus_{i=1}^r \widehat{M}_{ii}\otimes \Id_{\bbC^{\hat{v}_i}} }.
\label{cqac.8}\end{equation}

Notice that $( \oplus_{i=1}^3 I_iT_m\cO_m)\oplus\bold{T}_m$ can be identified with the tangent space at $\bold{q}(m)$ of the stratum of $M/G$ containing $\bold{q}(m)$. Let $\bold{T}_m^{\perp}$ be the orthogonal complement of $\bold{T}_m$ in $\widehat{\bold{M}}_m$.  Since the action of $G_m$ is trivial on $\bold{T}_m$, the moment map $\hat{\mu}_m$ is trivial on $\bold{T}_m$, that is, it factors through the projection 
\begin{equation}
        \widehat{\bold{M}}_m= \bold{T}_m\oplus \bold{T}_m^{\perp}\to \bold{T}_m^{\perp}
\label{cqac.8b}\end{equation}
and can be seen as a moment map on $\bold{T}_m^{\perp}$,
\begin{equation}
       \hat{\mu}_m: \bold{T}^{\perp}_m\to \bbR^3\otimes \mathfrak{g}_m^*.
\label{mm.1}\end{equation}
In fact, $\bold{T}^{\perp}_m$ is naturally a reduced Nakajima quiver representation space, namely
\begin{equation}
   \bold{T}_m^{\perp}= \bold{M}(\hat{v},\hat{w})_{\reduced}
\label{cqac.9}\end{equation}
with $\hat{v}={}^t(\hat{v}_1,\ldots,\hat{v}_r)$, $\hat{w}={}^t(\dim_{\bbC}\widehat{W}_1,\ldots,\dim_{\bbC}\widehat{W}_r)$ and with adjacency matrix $\hat{A}_m$ of the graph having $(i,j)$-component $\dim_{\bbC}\widehat{M}_{ij}$.  From the description above, $\bold{T}^{\perp}_m$ is the orthogonal complement of the tangent space at $m$ of the stratum of $\bold{M}$ containing $m$.  Clearly, the natural $\bbR^+$-action on $\bold{M}$ induces a canonical identification 
$$
       \bold{T}^{\perp}_{\lambda m}=\bold{T}^{\perp}_{ m} \quad \forall \lambda\in \bbR^+, \; \forall m\in \bold{M}\setminus\{0\}.
$$

For $\zeta\in \bbR^3\otimes Z \subset \bbR^3\otimes \mathfrak{g}^*$, denote by $\zeta_m$ its image under the canonical map
\begin{equation*}
\pr_m: \bbR^3\otimes \mathfrak{g}^*\to \bbR^3\otimes \mathfrak{g}^*_m.
\end{equation*}
Obviously, $\zeta_m$ is in the center of $\mathfrak{g}_m^*$, so $G_m$ acts on the preimage $\hat{\mu}_m^{-1}(-\zeta_m)$ of the moment map \eqref{mm.1}.  By the discussion above, the quotient $\hat\mu^{-1}(-\zeta_m)/G_m$ is a (reduced) Nakajima quiver variety.

In order for the arguments of section \ref{rcqv.0} to work, we need to assume that $\hat\mu^{-1}(-\zeta_m)/G_m$ is smooth for any choice of $m\in \mu^{-1}(0)\backslash\{0\} \subset \bold{M}\backslash\{0\}$ with non-trivial stabilizer $G_m$. From Theorem \ref{generic smoothness} this condition is equivalent to requiring that $\zeta\notin Z_s(\mathfrak g)$ and that $\forall m \in \mu^{-1}(0)\backslash\{0\}$ such that $G_m\ne \{\Id\}$, $\zeta_m\notin Z_s(\mathfrak g_m)$. 
\begin{definition}
We call $\zeta$ \textbf{properly generic} if it satisfies the above condition.
\label{pg.2}\end{definition}
\begin{lemma}
The subset of $\bbR^3\otimes Z$ where $\zeta$ fails to be properly generic is of real co-dimension three.
\label{propgen.1}\end{lemma}
\begin{proof}
We first need to properly define the sets $Z_s(\mathfrak g_m)$ as we did for the center of the full Lie algebra $\mathfrak g$ in \eqref{gen.1}. Given $m$ as above, the Cartan matrix is given by $C_m = 2\Id -\hat A_m$ with $\hat A_m$ the adjacency matrix of the reduced quiver variety \eqref{cqac.9}, while $R_+(\bold v)$ is replaced by $R_+(\hat v)$. By \eqref{cqac.7} and \cite[Lemma~6.5 (6)]{Nakajima1994} , the center of $G_m$ is $Z_m=\prod_{i=1}^{r}U_m^i(1)$ with 
$$
U_m^i(1):= G_m\cap \prod\limits_k U(V_k^{(i)})\cong U(1).
$$ 
Its Lie algebra is therefore given by $\displaystyle \mathfrak{z}_m=\bigoplus_{i=1}^r \mathfrak{u}_m^i(1)$ with $\mathfrak{u}_m^i(1)$ the Lie algebra of $U_m^i(1)$.  Consequently,  $\displaystyle Z_s(\mathfrak g_m):= \bigcup_{\theta\in R_+(\hat{v})}\bbR^3\otimes D^m_{\theta}$ with
\begin{equation*}
D^m_{\theta}:=  \left\{ x = (x_i) \in \bigoplus_{i=1}^r \mathfrak u^i_m(1) \Big| \sum\limits_k x_k\theta_k =0 \right\}\textrm{ for } \theta\in R_+(\hat v).
\end{equation*}
As is clear from \eqref{cqac.7} and the construction of the decomposition \eqref{cqac.5}, the subset $Z_s(\mathfrak g_m)$ only depends on $m$ through its stabilizer $G_m$, namely $Z_s(\mathfrak g_m)=Z_s(\mathfrak g_{m'})$ whenever $\mathfrak{g}_m=\mathfrak{g}_{m'}$.  

However, in a given conjugacy class of stabilizer subgroups, the subset $Z_s(\mathfrak g_m)$ may vary, in particular along the orbit $\cO_m$ of $m$.  Nevertheless, along such an orbit, the conditions $\zeta_{m'}\in Z_s(\mathfrak g_{m'})$ for $m'\in \cO_m$ only correspond to one condition, since by construction,
\begin{equation}
\zeta_m\in Z_s(\mathfrak g_m) \Leftrightarrow \zeta_{gm}\in Z_s(\mathfrak g_{gm}) \quad \forall g\in G.
\label{cong.1}\end{equation}
  Indeed, since moving from $m$ to $gm$ changes the whole decomposition \eqref{cqac.5} by its composition by $g$, the Cartan matrix and $R_+(\hat v)$ are the same for $m$ and $gm$. Since $G_{gm} = gG_mg^{-1}$, it follows that $\mathfrak u_{gm}^i(1) = g\mathfrak u_m^i(1)g^{-1}$ for all $i$ and therefore that  $D^{gm}_{\theta} = gD^m_{\theta}g^{-1}$ and $Z_s(\mathfrak g_{gm}) = gZ_s(\mathfrak g_m)g^{-1}$. Hence, \eqref{cong.1} follows from the fact that $\mathfrak g_{gm} = g\mathfrak g_m g^{-1}$, so in particular $\zeta_{gm} = g\zeta_m g^{-1}$. As a result, the condition that $\zeta_m\in Z_s(\mathfrak g_m)$ only depends on the conjugacy class of $G_m$. 
  
  Since there are only finitely many conjugacy classes of stabilizer subgroups in $G$,  the result follows provided we can show that $\pr_m^{-1}(\bbR^3\otimes Z_s(\mathfrak g_m))$ is of real co-dimension three inside $\bbR^3\otimes Z$.  To see this, let  $\mathfrak{u}(1)$ and $\mathfrak{u}_m(1)$ denote the Lie algebras of scalars for $G$ and $G_m$.  By \cite[Lemma~6.5 (6)]{Nakajima1994}, the restriction of $\pr_m$ to $\bbR^3\otimes \mathfrak{u}(1)\subset \bbR^3\otimes Z$ induces a surjective map
  $$
      \pr_m: \bbR^3\otimes\mathfrak{u}(1)\to \bbR^3\otimes\mathfrak{u}_m(1).
  $$
  On the other hand, by \eqref{gen.1}, an element $\zeta_m\in \bbR^3\otimes\mathfrak{u}_m(1)$ is in $Z_s(\mathfrak g_m)$ if and only if $\zeta_m=0$.  This shows that $\bbR^3\otimes Z_s(\mathfrak g_m)\cap \pr_m(\bbR^3\otimes Z)$ is of real codimension three in  $\pr_m(\bbR^3\otimes Z)$, hence that $\pr_m^{-1}(\bbR^3\otimes Z_s(\mathfrak g_m))$ is of real co-dimension three inside $\bbR^3\otimes Z$ has desired.

\end{proof}

In the decomposition 
\begin{equation}
  T_m\bold{M}= (\bbH T_m\cO_m)\oplus \widehat{\bold{M}}_m,
\label{cor.1}\end{equation}
the factor $\bbH T_m\cO_m$ will ultimately play no significant role after passing to the hyperK\"ahler quotient.  However, the action of $G_m$ on $\bbH T_m\cO_m$ is not trivial in general and we need to carefully describe it.  First, since the action of $G_m$ preserves the $\bbH$-module structure of $\bbH T_m\cO_m$, it suffices to describe its action on $T_m\cO_m$.  Using the identification
$$
   T_m\cO_m= \mathfrak{g}^{\perp}_m,
$$
this action is given by the adjoint action of $G_m$ on $\mathfrak{g}^{\perp}_m$.  This action is originally defined on $\mathfrak{g}$, but since it defines an orthogonal representation of $G_m$ and preserves $\mathfrak{g}_m$ in the orthogonal decomposition 
$$
     \mathfrak{g}= \mathfrak{g}_m\oplus \mathfrak{g}_m^{\perp},
$$
it induces an action on $\mathfrak{g}^{\perp}_m$ as well.  Let $\mathfrak{g}^{\perp}_{m,0}$ be te subspace of $\mathfrak{g}^{\perp}_m$ consisting of elements fixed by the adjoint action of $G_m$ and let $\mathfrak{g}_{m,1}^\perp$ be its orthogonal complement in $\mathfrak{g}^{\perp}_m$.  This induces the decomposition
$$
  \bbH T_m\cO_m= \bbH \mathfrak{g}^{\perp}_{m,0} \oplus \bbH \mathfrak{g}^{\perp}_{m,1}
$$
with $G_m$ acting trivially on the first factor.  Inserting this in \eqref{cor.1}, this induces via \eqref{cqac.8b} the decomposition
\begin{equation}
   T_m \bold{M}= \bbH \mathfrak{g}^{\perp}_{m,0}\oplus \bold{T}_m \oplus \bbH \mathfrak{g}^{\perp}_{m,1} \oplus \bold{T}_m^{\perp}.
\label{cor.2}\end{equation}
Upon making the identification $T_m \bold{M}= \bold{M}$, the subspace $\bbH \mathfrak{g}^{\perp}_{m,0}\oplus \bold{T}_m$ corresponds to the subspace $\bold{M}^{G_m}$ of $\bold{M}$ consisting of the elements fixed by the action of $G_m$, so that \eqref{cor.2} can be rewritten
\begin{equation}
  T_m\bold{M}= \bold{M}^{G_m}\oplus (\bbH \mathfrak{g}^{\perp}_{m,1}\oplus \bold{T}^{\perp}_m).
\label{cor.3}\end{equation}
By definition of $G_m$, notice that $m\in \bold{M}^{G_m}\subset \bold{M}$.  In fact, $\bold{M}^{G_m}$ contains all the elements with stabilizer $G_m$, but some of its elements may have strictly larger stabilizers.

We can now introduce the natural compactification of $\bold{M}$ to describe the asymptotic geometry at infinity of the associated quiver variety.  Let us first denote by $\overline{\bold{M}}$ the radial compactification of $\bold{M}$.    Let $\mathfrak{s}_1,\ldots,\mathfrak{s}_{\ell}$ be the strata of $\bbS(\bold{M})=\pa\overline{\bold{M}}$ as a $G$-manifold listed in an order compatible with the partial order, namely
$$
    \mathfrak{s}_i<\mathfrak{s}_j \; \Longrightarrow \; i<j.
$$ 
In particular, $\mathfrak{s}_{\ell}$ corresponds to the regular stratum of $\pa\overline{\bold{M}}$.  
\begin{definition}
The \textbf{$\QAC$ compactification} of $\bold{M}$ seen as an orthogonal representation of $G$ is the manifold with corners 
$$
      \widetilde{\bold{M}}:= [\overline{\bold{M}}; \overline{\mathfrak{s}}_1,\ldots,\overline{\mathfrak{s}}_{\ell-1}].
$$
\label{cqac.10}\end{definition}
In this definition, the order in which we blow up is important.  First, since $\mathfrak{s}_1$ is minimal with respect to the partial order, $\mathfrak{s}_1=\overline{\mathfrak{s}}_1$ is a closed submanifold of $\pa\overline{M}$, so its blow-up is well-defined, as well as the blow-ups of all minimal strata.  More generally, before the blow-up of $\overline{\mathfrak{s}}_i$ is performed, notice that $\overline{\mathfrak{s}}_{j}$ has been blown up whenever $\mathfrak{s}_j<\mathfrak{s}_i$, so by \cite[Proposition~7.4 and Theorem~7.5]{AM2011}, the lift of $\overline{\mathfrak{s}}_i$ is a $p$-submanifold and its blow-up is well-defined.  The manifold with corners $\widetilde{\bold{M}}$ has $\ell$ boundary hypersurfaces $H_1,\ldots, H_{\ell}$ corresponding to the strata $\mathfrak{s}_1,\ldots,\mathfrak{s}_{\ell}$ of $\pa\overline{\bold{M}}$.  By \cite[Theorem~7.5]{AM2011}, the maximal hypersurface $H_{\ell}$ has an iterated fibration structure, namely it is a manifold with fibered corners. Clearly, this iterated fibration naturally extends to induce on $\widetilde{M}$ an iterated fibration structure with fiber bundle 
\begin{equation}
        \phi_{H_i}: H_i\to S_{H_i}
\label{cqac.11}\end{equation}
induced by the blow-down map
$$
       \widetilde{\bold{M}}\to [\overline{\bold{M}};\overline{\mathfrak{s}}_1,\ldots,\overline{\mathfrak{s}}_{i-1}],
$$
where $S_{H_i}$ is the manifold with fibered corners resolving the smoothly stratified space $\overline{\mathfrak{s}}_i$.  To describe the fibers of \eqref{cqac.11}, consider first the case $i=1$ and let $\overline{H}_1$ be the boundary hypersurface in $[\overline{\bold{M}};\overline{\mathfrak{s}_1}]$ created by the blow-up of $\overline{\mathfrak{s}}_1$.  Then the blow-down map $[\overline{\bold{M}};\overline{\mathfrak{s}}_1]\to \overline{\bold{M}}$ induces a fiber bundle 
$$
     \overline{\phi}_{H_1}: \overline{H_1}\to S_{H_1}=\overline{\mathfrak{s}}_1.
$$
By the discussion above, the fibers above $m\in S_{H_1}$ correspond to the radial compactification $\overline{\bbH\mathfrak{g}^{\perp}_{m,1}\oplus \bold{T}^{\perp}_m}$ of the space orthogonal to the tangent space of the stratum of $\pa\overline{\bold{M}}$ containing $m$.  Now, when we perform the blow-ups of the other strata of $\pa\overline{\bold{M}}$, this corresponds on $\overline{\bbH\mathfrak{g}^{\perp}_{m,1}\oplus\bold{T}^{\perp}_m}$ to the blow-ups of the intersection of the  lifts of the strata $\overline{\mathfrak{s}}_2,\ldots \overline{\mathfrak{s}}_{\ell-1}$ with $\pa\overline{\bbH\mathfrak{g}^{\perp}_{m,1}\oplus\bold{T}^{\perp}_m}$, so that on $\widetilde{\bold{M}}$, $\overline{\bbH\mathfrak{g}^{\perp}_{m,1}\oplus\bold{T}^{\perp}_m}$ lifts to the $\QAC$ compactification $\widetilde{\bbH\mathfrak{g}^{\perp}_{m,1}\oplus\bold{T}^{\perp}_m}$ of  $\bbH\mathfrak{g}^{\perp}_{m,1}\oplus\bold{T}^{\perp}_m$ as an orthogonal representation of $G_m$.  This $\QAC$ compactification contains the $\QAC$ compactification $\widetilde{\bold{T}^{\perp}_m}$ of $\bold{T}^{\perp}_m$, namely as the closure of $\{0\}\times \bold{T}^{\perp}_m$ in $\widetilde{\bbH\mathfrak{g}^{\perp}_{m,1}\oplus\bold{T}^{\perp}_m}$.   For the other boundary hypersurfaces of $\widetilde{\bold{M}}$, the same phenomenon occurs if $m\in S_{H_i}\setminus \pa S_{H_i}$, namely the fiber above $m$ corresponds to the $\QAC$ compactification $\widetilde{\bbH\mathfrak{g}^{\perp}_{m,1}\oplus \bold{T}^{\perp}_m}$ of $\bbH\mathfrak{g}^{\perp}_{m,1}\oplus \bold{T}^\perp_m$.  When $H_i$ is not minimal and $m\in \pa S_{H_i}$, this is also what happens since $\bbH\mathfrak{g}^{\perp}_{m,1}\oplus \bold{T}^{\perp}_m$ is now the orthogonal of the complement of $T_mS_{H_i}$ in the tangent space at $m$ in 
$$
         [\pa \overline{\bold{M}}; \overline{\mathfrak{s}}_1,\ldots,\overline{\mathfrak{s}}_{i-1}]
$$
with $G_m$ the stabilizer of $m$ in $[\pa \overline{\bold{M}}; \overline{\mathfrak{s}}_1,\ldots,\overline{\mathfrak{s}}_{i-1}]$.  

For $\zeta \in \bbR^3\otimes Z$, let $\overline{\mu^{-1}(-\zeta)}$ and $\widetilde{\mu^{-1}(-\zeta)}$ be the closure of $\mu^{-1}(\zeta)$ inside $\overline{\bold{M}}$ and $\widetilde{\bold{M}}$ respectively. Since $\mu$ is homogeneous of degree $2$ with respect to the natural $\bbR^+$-action on $\bold{M}$, notice that $\mu^{-1}(0)$ is a cone in $\bold{M}$ with possibly a singular cross-section, while for $\zeta\in \bbR^3\otimes Z$ fixed, $\mu^{-1}(-\zeta)$ is asymptotic to $\mu^{-1}(0)$ at infinity in the sense that 
\begin{equation}
       \overline{\mu^{-1}(-\zeta)}\cap\pa \overline{\bold{M}}= \overline{\mu^{-1}(0)}\cap\pa \overline{\bold{M}}.
\label{cqac.12}\end{equation}
Hence, unless the cone $\mu^{-1}(0)$ has a smooth cross-section, even if $\mu^{-1}(-\zeta)$ is smooth, for instance when $\zeta$ is generic, $\overline{\mu^{-1}(-\zeta)}$ will not be smooth and will have singularities on its boundary $\pa\overline{\mu^{-1}(-\zeta)}=\overline{\mu^{-1}(0)}\cap\pa \overline{\bold{M}}$.  Now, by the proof of Lemma~\ref{cqac.1}, the cone $\mu^{-1}(0)$ is naturally stratified with strata induced by those of the $G$-space $\bold{M}$.  Similarly, $\overline{\mu^{-1}(0)}$ and $\pa\overline{\mu^{-1}(0)}=\overline{\mu^{-1}(0)}\cap\pa \overline{\bold{M}}$ are stratified by the strata induced by those of $\overline{\bold{M}}$ and $\pa \overline{\bold{M}}$.  In particular, for each stratum $\mathfrak{s}_i\subset \pa\overline{\bold{M}}$, there is a corresponding stratum $\mathfrak{s}_i\cap \pa\overline{\mu^{-1}(0)}$ of $\pa\overline{\mu^{-1}(0)}$.  On $\widetilde{\bold{M}}$, we can correspondingly associate to each $H_i\in \cM_1(\widetilde{\bold{M}})$ a `boundary hypersurface' $H_i\cap \widetilde{\mu^{-1}(0)}$ of $\widetilde{\mu^{-1}(0)}$ with 

\begin{equation}
    \Sigma_{H_i}:= \phi_{H_i}(H_i\cap \widetilde{\mu^{-1}(0)})\subset S_{H_i}
\label{cqac.13}\end{equation}
a subset of $S_{H_i}$ mapping onto  the closed stratum $\overline{\mathfrak{s}}_i\cap\pa\overline{\mu^{-1}(0)}$ under the blow-down map $S_{H_i}\to \overline{\mathfrak{s}}_i$.  
\begin{theorem}
For $\zeta\in \bbR^3\otimes Z$ properly generic, $\widetilde{\mu^{-1}(-\zeta)}$ is a $p$-submanifold of $\widetilde{\bold{M}}$ such that the iterated fibration structure of $\widetilde{\bold{M}}$ induces one on $\widetilde{\mu^{-1}(-\zeta)}$, namely for each $H_i\in\cM_1(\widetilde{\bold{M}})$, $\widetilde{\mu^{-1}(-\zeta)}$ has a boundary hypersurface $H_i\cap \widetilde{\mu^{-1}(-\zeta)}$ with fiber bundle
\begin{equation}
     \phi_{H_i}: H_i\cap \widetilde{\mu^{-1}(-\zeta)}\to \Sigma_{H_i}
\label{cqac.14a}\end{equation}
induced by restriction of the fiber bundle $\phi_{H_i}:H_i\to S_{H_i}$.  
Moreover, the free $G$-action on $\mu^{-1}(-\zeta)$ extends to a free $G$-action on $\widetilde{\mu^{-1}(-\zeta)}$ in such a way that for each $H_i\in\cM_1(\widetilde{\bold{M}})$, there is an induced action on $\Sigma_{H_i}$ making \eqref{cqac.14a} $G$-equivariant.
\label{cqac.14}\end{theorem}
\begin{proof}
Proceeding by induction on the depth of the Nakajima quiver representation space $\bold{M}$ as a $G$-space, we can assume that the result holds for Nakajima quiver representation spaces of lower depth.  Let us first consider the case where $H_i$ is minimal.  Given $m\in \Sigma_{H_i}\subset S_{H_i}$, the corresponding fiber $\phi_{H_i}^{-1}(m)$  in $H_i$ is the $\QAC$ compactification $\widetilde{\bbH\mathfrak{g}_{m,1}^{\perp}\oplus\bold{T}^{\perp}_m}$.  In terms of the decomposition 
\begin{equation}
\bbR^3\otimes\mathfrak{g}^*= (\bbR^3\otimes\mathfrak{g}_m^*)\oplus(\bbR^3\otimes(\mathfrak{g}_m^{\perp})^*),
\label{cqac.15}\end{equation}  
the moment $\mu$ has a decomposition
\begin{equation}
\mu= \hat{\mu}_m+ \check{\mu}_m
\label{cqac.16}\end{equation}
with 
\begin{equation}
  \hat{\mu}_m:\bold{M}\to \bbR^3\otimes\mathfrak{g}_m^* \quad \mbox{and} \quad  \check{\mu}_m: \bold{M}\to \bbR^3\otimes(\mathfrak{g}^{\perp}_m)^*.
\label{cqac.17}\end{equation}
The equation $\mu(q)=-\zeta$ then becomes
\begin{equation}
      \hat{\mu}_m(q)=-\zeta_m \quad \mbox{and} \quad \check{\mu}_m(q)=-\zeta_{m}^{\perp},
\label{cqac.17f}\end{equation}
where $\zeta=(\zeta_m,\zeta_m^{\perp})$ in the decomposition \eqref{cqac.15}.  

We will use this decomposition to see that $\widetilde{\mu^{-1}(-\zeta)}$ is a $p$-submanifold near $H_i$, but in order to do that, we need to introduce suitable coordinates.  Let $\bbS(\bold{M})$ be the sphere of radius one centered at the origin in $\bold{M}$.  Without loss of generality we can think of $m$ as a point on that sphere.  Since the action of $G$ is unitary, its action restricts to $\bbS(\bold{M})$.  Applying the tube theorem \cite[Theorem~2.4.1]{DK2000} near $m$ on $\bbS(\bold{M})$, we can find a local neighborhood $\cU$ of $m$ in $\bbS(\bold{M})$ with an equivariant diffeomorphism
$$
   \cU\cong G\times_{G_m} V
$$
where $G_m$ the stabilizer of $m$ and $V$ is the orthogonal representation of $G_m$ corresponding to the orthogonal complement of $T_m\cO_m$ in $T_m\bold{M}$, namely
\begin{equation}
\begin{aligned}
  V &= (\oplus_{i=1}^3 I_iT_m\cO_m)\oplus \widehat{\bold{M}}_m = (\oplus_{i=1}^3 I_i \mathfrak{g}^{\perp}_{m,0})\oplus (\oplus_{i=1}^3 I_i \mathfrak{g}^{\perp}_{m,1}) \oplus \bold{T}_m\oplus \bold{T}^{\perp}_m \\
  &\cong \left( (\bbR^3\otimes \mathfrak{g}^\perp_{m,0}) \oplus \bold{T}_m \right) \oplus \left( (\bbR^3\otimes \mathfrak{g}^\perp_{m,1}) \oplus \bold{T}_m^{\perp} \right)
 \end{aligned} 
\label{cor.4}\end{equation}
with $(\bbR^3\otimes \mathfrak{g}^\perp_{m,0}) \oplus \bold{T}_m$ identified with $V^{G_m}$.
 
 Let
$$
 \cC_{\cU}:= \{ \lambda p\; |\; p\in \cU, \; \lambda\ge 0\}
$$
be the corresponding cone in $\bold{M}$ with cross-section $\cU$.  Since $\mu$ is equivariant and $\zeta\in\bbR^3\otimes Z$ with $Z$ the center of $\mathfrak{g}$,  to check that $\widetilde{\mu^{-1}(-\zeta)}\cap \cC_{\cU}$ is a $p$-submanifold near $H_i$, it suffices to check that  $\widetilde{\mu^{-1}(-\zeta)}\cap \cC_{V}$ is a $p$-submanifold near $H_i$, where $\cC_V$ is the cone over $V$ seen as the subset $\{e\}\times V$ of $\cU$.   The invariant subspace  $V^{G_m}\cong (\bbR^3\otimes \mathfrak{g}^\perp_{m,0}) \oplus \bold{T}_m $ can be identified with the stratum $\overline{\mathfrak{s}}_i\cap \overline{\cC}_V$.  On the other hand, $\phi_{H_i}^{-1}(m)\cap \widetilde{\cC_V}$ corresponds to the $\QAC$ compactification of $(\bbR^3\otimes \mathfrak{g}^\perp_{m,0}) \oplus \bold{T}_m^{\perp}$ as an orthogonal representation of $G_m$.

Let $\check{\mu}_m= \check{\mu}_{m,0}+  \check{\mu}_{m,1}$ be the decomposition of $\check{\mu}_m$ in terms of the decomposition 
\begin{equation}
\bbR^3\otimes (\mathfrak{g}^{\perp}_m)^*= \bbR^3\otimes (\mathfrak{g}^{\perp}_{m,0})^*\oplus  \bbR^3\otimes (\mathfrak{g}^{\perp}_{m,1})^*.
\label{cor.5}\end{equation}
Then $\varpi_0:= \check{\mu}_{m,0}|_V$ and $\varpi_1:= \check{\mu}_{m,1}|_V$ can be seen as coordinates on the factors $ \bbR^3\otimes (\mathfrak{g}^{\perp}_{m,0})^*$ and $ \bbR^3\otimes (\mathfrak{g}^{\perp}_{m,1})^*$ in the decomposition \eqref{cor.4}.  Let also $\check{\varpi}$ and $\hat{\varpi}$ be choices of coordinates on $\bold{T}_m$ and $\bold{T}_m^\perp$, so that
$$
   \varpi=(\varpi_0,\check{\varpi},\varpi_1,\hat{\varpi})
$$
are coordinates on $V$ in terms of the decomposition \eqref{cor.4}.

 Now, if $\mu_V$ is the restriction of the moment map to $V$, then by homogeneity, on $\mathcal{C}_V$, $\mu$ is given by
$$
    \mu(\rho,\varpi)=\rho^2\mu_V(\varpi)
$$  
for $\varpi\in V$ and $\rho$ the distance function from the origin in $\bold{M}$.  On the other hand, in the coordinates $\varpi$, the  blow-up of $\overline{s}_i$ at infinity corresponds to introducing the coordinates $(u, \varpi_0,\check{\varpi} , v_1, \hat{v})$ with $u=\rho^{-1}$, $v_1=\frac{\varpi_1}{u}$ and $\hat{v}= \frac{\hat{\varpi}}{u}$.  In these coordinates, the fiber bundle $\phi_i: H_i\to S_i$ corresponds to the projection
$$
   ( \varpi_0,\check{\varpi} , v_1, \hat{v})\mapsto (\varpi_0,\check{\varpi}).
$$
If $\zeta^{\perp}_m= \zeta_{m,0}^{\perp}+ \zeta^{\perp}_{m,1}$ in terms of the decomposition \eqref{cor.5}, then in the coordinates $(u, \varpi_0,\check{\varpi} , v_1, \hat{v})$, the equations \eqref{cqac.17f} take the form
\begin{gather}
\label{cor.5a}  \varpi_0= -u^2\zeta^{\perp}_{m,0}, \\
\label{cor.5b} v_1= -u\zeta^{\perp}_{m,1}, \\
\label{cor.5c} \hat{\mu}_m(\hat{v}) + \hat{\mu}_m(v_1)= -\zeta_m.
\end{gather}
Notice that substituting \eqref{cor.5b} in \eqref{cor.5c} yields
\begin{equation}
   \hat{\mu}_m(\hat{v}) + u^2\hat{\mu}_m(-\zeta_{m,1}^{\perp})= -\zeta_m,
\label{cor.5d}\end{equation}
where $-\zeta^{\perp}_{m,1}$ is seen as fixed value of the coordinate $v_1$.  Now, equations \eqref{cor.5a}, \eqref{cor.5b} and \eqref{cor.5d} makes sense at $u=0$, in which case we obtain
\begin{gather}
\label{cor.6a} \varpi_0=0, \\
\label{cor.6b} v_1=0, \\
\label{cor.6c}  \hat{\mu}_m(\hat{v})= -\zeta_m.
\end{gather}
Equation \eqref{cor.6a} defines $\Sigma_{H_i}$ as a smooth submanifold of $S_{H_i}$ with no constraint on $\check{\varpi}$.  The other two equations defines $\widetilde{\mu^{-1}(-\zeta)}\cap \widetilde{\cC_V}\cap \phi_{H_i}^{-1}(m)$ in the interior of $\phi_{H_i}^{-1}(m)\setminus \pa \phi_{H_i}^{-1}(m)$.  The equation \eqref{cor.6b} indicates that $\widetilde{\mu^{-1}(-\zeta)}\cap \widetilde{\cC_V}\cap \phi_i^{-1}(m)$ lies in the $\QAC$ compactification $\widetilde{\bold{T}^{\perp}_m}$ included in the $\QAC$ compactification of $\bbR^3\otimes \mathfrak{g}_{m,1}^{\perp}\oplus \bold{T}^{\perp}_m$ and corresponds to $\widetilde{\hat{\mu}_m^{-1}(-\zeta_m)}$ inside  $\widetilde{\bold{T}^{\perp}_m}$.  In particular, since $\zeta$ is properly generic, $\widetilde{\hat{\mu}_m^{-1}(-\zeta_m)}$ is smooth in its interior.  Moreover, by induction on the depth, we can assume that $\widetilde{\hat{\mu}_m^{-1}(-\zeta_m)}$ is a $p$-submanifold inside $\widetilde{\bold{T}^{\perp}_m}$ on which $G_m$ acts freely.  The local description \eqref{cor.5a},\eqref{cor.5b} and \eqref{cor.5d} then shows that $\widetilde{\mu^{-1}(-\zeta)}$ is a p-submanifold near $\phi_{H_i}^{-1}(m)$ with a natural fiber bundle  
$$
H_i\cap\widetilde{\mu^{-1}(-\zeta)}\to \Sigma_{H_i}
$$
induced by $\phi_{H_i}$.  Since $G_m$ acts freely on $\widetilde{\hat{\mu}_m^{-1}(-\zeta_m)}$, the action of $G$ extends to a free $G$-action on $\widetilde{\mu^{-1}(-\zeta)}$ near $\phi_{H_i}^{-1}(m)$.  Since $m\in \Sigma_{H_i}$ was arbitrary, we see that the result holds near $H_i$.     

Along a non-minimal boundary hypersurface $H_i$,  we can assume by induction that we already know that $\widetilde{\mu^{-1}(-\zeta)}$ is a $p$-submanifold near $H_j$ for $H_j<H_i$.  Thus, it suffices to show that $\widetilde{\mu^{-1}(-\zeta)}$ is a $p$-submanifold near $\phi_{H_i}^{-1}(m)$ for $m$ in the interior of $\Sigma_{H_i}$, so that the same argument as before applies.  Clearly then $\widetilde{\mu^{-1}(-\zeta)}$ is smooth and has the claimed iterated fibration structure making it a manifold with fibered corners.  Moreover, the free action of $G$ on $\mu^{-1}(-\zeta)$ extends to a free smooth action on $\widetilde{\mu^{-1}(-\zeta)}$.  Clearly, for each $H_i\in\cM_1(\widetilde{\bold{M}})$, there is an induced action on $\Sigma_{H_i}$ making the fiber bundle \eqref{cqac.14a} $G$-equivariant. 
\end{proof}
\begin{remark}
In particular, notice that Theorem~\ref{cqac.14} shows that $\Sigma_{H_i}$ in \eqref{cqac.13} is a $p$-submanifold of $S_{H_i}$ with an induced iterated fibration structure.  Theorem~\ref{cqac.14} also implies that $\pa \overline{\mu^{-1}(0)}$ is a smoothly stratified space with resolution the manifold with fibered corners $H_\ell\cap \widetilde{\mu^{-1}(-\zeta)}= H_\ell\cap \widetilde{\mu^{-1}(0)}$.  By homogeneity, the subset $\mu^{-1}(0)$ is also a smoothly stratified space.
\label{cor.7}\end{remark}
\begin{remark}
For each $H_i\in \cM_1(\widetilde{\bold{M}})$, notice that the induced action of $G$ on $\Sigma_{H_i}$ has only one conjugacy class of stabilizer subgroups, namely the one associated to $\mathfrak{s}_i$.  In particular, the quotient $\Sigma_{H_i}/G$ is a manifold with corners and the iterated fibration structure of $\widetilde{\mu^{-1}(-\zeta)}$ induces one on the quotient $\widetilde{\mu^{-1}(-\zeta)}/G$.  For this reason, we say that the action of $G$ on $\widetilde{\mu^{-1}(-\zeta)}$ is \textbf{compatible} with the iterated fibration structure.
\label{cor.7b}\end{remark}
\begin{corollary}
For $\zeta\in \bbR^3\otimes Z$ properly generic, the Nakajima metric on the quiver variety $\mathfrak{M}_{\zeta}$ is an exact quasi-asymptotically conical metric with smooth expansion at infinity.
\label{cqac.21}\end{corollary}
\begin{proof}
By \cite{CDR}, the Euclidean metric on $\bold{M}$ can be seen as an exact $\QAC$ metric on $\widetilde{\bold{M}}$. Since the iterated fibration structure of $\widetilde{\mu^{-1}(-\zeta)}$ is induced by the one of $\widetilde{\bold{M}}$ through the inclusion $\widetilde{\mu^{-1}(-\zeta)}\to \widetilde{\bold{M}}$, the restriction of the Euclidean metric of $\bold{M}$ to $\widetilde{\mu^{-1}(-\zeta)}$ is automatically an exact $\QAC$ metric.  Since $G$ acts freely on $\widetilde{\mu^{-1}(-\zeta)}$ in a way compatible with the metric and the iterated fibration structure, this metrics descends to induce an exact $\QAC$ metric on the quotient $\widetilde{\mu^{-1}(-\zeta)}/G$ with smooth asymptotic expansion at infinity.
\end{proof}

\section{$L^2$-cohomology of incomplete metrics}\label{l2i.0}

This section will recall basic facts about the $L^2$-cohomology of incomplete Riemannian metrics, notably about the $L^2$-cohomology of wedge metrics.  We will also introduce a $L^2$-K\"ahler package for such metrics.  It is weaker than the one of \cite{BL93}, but has the advantage of giving the version of the Hard Lefschetz theorem that we need for the class of wedge K\"ahler metrics we will consider later on.   Let us first recall the notion of Hilbert complexes introduced in \cite{BL92}.
\begin{definition}
A \textbf{Hilbert complex} is a sequence
\begin{equation}
 \xymatrix{
 0 \ar[r] & L_0 \ar[r]^{D_0} & L_1 \ar[r]^{D_1}\ar[r]& \cdots \ar[r]^{D_{n-1}} & L_n \ar[r] & 0 
 }
\label{l2i.1a}\end{equation}
with each $L_i$ a separable Hilbert space and $D_i: L_i\to L_{i+1}$ a closed operator with dense domain $\cD(D_i)$ such that $\Im (D_i)\subset \cD(D_{i+1})$ and $D_{i+1}\circ D_i=0$.  Thus, though \eqref{l2i.1a} is not properly speaking a complex in general, it induces a complex
\begin{equation}
 \xymatrix{
 0 \ar[r] & \cD(D_0) \ar[r]^{D_0} & \cD(D_1) \ar[r]^{D_1}\ar[r]& \cdots \ar[r]^{D_{n-1}} & \cD(D_n) \ar[r] & 0. 
 }
\label{l2i.1b}\end{equation}
\label{l2i.1}\end{definition}
There is a nice Hodge theory attached to such a Hilbert complex, namely there is a natural dual Hilbert complex
\begin{equation}
\xymatrix{
  0  & L_0 \ar[l] & L_1 \ar[l]_{D_0^*} & L_2 \ar[l]_{D_1^*} & \cdots \ar[l]_{D_2^*} & L_n \ar[l]_{D_{n-1}^*} & \ar[l] 0 
}
\label{l2i.3}\end{equation}
with $D_k^*$ the adjoint of $D_k$, as well as an associated `Hodge Laplacian'
$$
   \Delta_k= D_k^*D_k + D_{k-1}D_{k-1}^* \quad \mbox{on} \; L_k
$$
with domain
$$
 \cD(\Delta_k)=\{ u\in \cD(D_k)\cap \cD(D_{k-1}^*)\; | \; D_ku\in \cD(D_k^*), \; D_{k-1}^*u\in \cD(D_{k-1})\}.
$$
By \cite[Lemma~2.1]{BL92}, there is a weak Kodaira decomposition
\begin{equation}
L_k= \ker\Delta_k \oplus \overline{\Im D_{k-1}}\oplus \overline{\Im D_{k}^*}.
\label{l2i.4}\end{equation}

In our setting, the Hilbert complexes will be induced by the exterior differential on a possibly incomplete oriented Riemannian manifold $(M,g)$ equipped with a flat Euclidean vector bundle $E\to M$.  Thus, our $L_k$ will be the space $L^2\Omega^k(M;E,g)$ of sections of $\Lambda^k(T^*M)\otimes E$ that are $L^2$ with respect to the $L^2$-norm induced by $g$ and the bundle metric on $E$.  These are separable Hilbert spaces with exterior differential densely defined on smooth forms of compact support.  In general however, it can admit different closed extensions.  We will consider the following two.  
\begin{definition}
The \textbf{minimal extension} $d_{\min,k}$ of the exterior differential on forms of degree $k$ is the graph closure of $d$ on $\Omega^k_c(M;E)$, namely
\begin{multline}
  \cD(d_{\min,k})= \{ \nu\in L^2\Omega^kM;E,g) \; | \; \exists \nu_j\in \Omega^k_c(M;E) \; \mbox{such that} \; \nu_j\to \nu \in L^2\Omega^k(M;E,g) \\ \mbox{and} \; \{d\nu_j\} \;\mbox{converges in $L^2$ to some } \; \eta\in L^2\Omega^{k+1}(M;E,g)\}.
\end{multline}
For such a $\nu$ with such a sequence $\{\nu_j\}$, we then have 
$$
d_{\min,k}\nu:= \lim_{j\to \infty} d\nu_j= \eta\in L^2\Omega^{k+1}(M;E,g). 
$$
On the other hand, the \textbf{maximal extension} $d_{\max,k}$ of $d$ on forms of degree $k$ is the closed extension with domain 
$$
   \cD(d_{\max,k})= \{\nu\in L^2\Omega^k(M;E,g) \; | \; d\nu\in L^2\Omega^*(M;E,g)\}.
$$
For such a $\nu\in \cD(d_{\max,k})$, $d_{\max,k}\nu= d\nu\in L^2\Omega^{k+1}(M;E,g)$.
\label{l2i.5}\end{definition}

On a complete oriented Riemannian manifold, these two closed extensions coincide by a result of Gaffney \cite{Gaffney1951} and there is in fact a unique closed extension.  In general however, these two extensions may differ.  A simple but important observation is that $d_{\min,k}$ and $d_{\max,k}$ only depend on the quasi-isometric class of the metric $g$.  They define Hilbert complexes 
\begin{gather}
\label{l2i.6a}  \xymatrix{ \cdots \ar[r]& L^2\Omega^k(M;E,g)\ar[r]^{d_{\min,k}} & L^2\Omega^{k+1}(M;E,g)\ar[r] & \cdots }, \\
\label{l2i.6b}  \xymatrix{ \cdots \ar[r]& L^2\Omega^k(M;E,g)\ar[r]^{d_{\max,k}} & L^2\Omega^{k+1}(M;E,g)\ar[r] & \cdots }.
\end{gather}
If $d^*$ is the formal adjoint of the exterior differential $d$, then it admits a minimal and a maximal extensions $d^*_{\min,k}$ and $d^*_{\max,k}$ on forms of degree $k+1$, so that $d^*_{\min,k}$ is the adjoint of $d_{\max,k}$ and $d^{*}_{\max,k}$ is the adjoint of $d_{\min,k}$.  This lead to two different Hodge Laplacians, namely the \textbf{relative Hodge Laplacian}
\begin{equation}
  \Delta_{\rel}:= d^*_{\max}d_{\min}+ d_{\min}d^{*}_{\max}
\label{l2i.7a}\end{equation}
associated to the Hilbert complex \eqref{l2i.6a}, where $d_{\min/\max}$  (respectively $d^{*}_{\min/\max}$) denotes the minimal/maximal extension of $d$ (respectively $d^*$), and the \textbf{absolute Hodge Laplacian}
\begin{equation}
 \Delta_{\abs}:= d^*_{\min}d_{\max}+ d_{\max}d^{*}_{\min}
\label{l2i.7b}\end{equation}
associated to the Hilbert complex \eqref{l2i.6b}.
\begin{definition}
The \textbf{minimal $L^2$-cohomology} of $(M,E,g)$ is the cohomology of the (complex associated to the) Hilbert complex \eqref{l2i.6a},
while the \textbf{maximal $L^2$-cohomology} of $(M,E,g)$ is the cohomology of the (complex associated to the) Hilbert complex \eqref{l2i.6b}.  We denote the corresponding cohomology groups of degree $k$ by respectively $L^2H^k_{\min}(M;E,g)$ and $L^2H^k_{\max}(M;E,g)$.  
\label{l2i.8}\end{definition}
\begin{remark}
 When $d$ is essentially self-adjoint, for instance when $g$ is complete, then $d_{\min}=d_{\max}$ and these cohomology groups agree, in which case we can denote them unambiguously by $L^2H^k(M;E,g)$. 
\label{l2i.8b}\end{remark}
These $L^2$-cohomology groups can often be infinite dimensional, in which case it can be useful to consider the minimal/maximal \textbf{reduced}   $L^2$-cohomology groups
\begin{equation}
 L^2_rH^k_{\min}(M;E,g):= \ker d_{\min,k}/ \overline{\Im d_{\min,k-1}} \quad \mbox{and} \quad L^2_rH^k_{\max}(M;E,g):= \ker d_{\max,k}/ \overline{\Im d_{\max,k-1}}.
\label{l2i.9}\end{equation}
Reduced or not, the minimal and maximal $L^2$-cohomology groups only depend on the quasi-isometric class of the metric $g$.  In general, the reduced minimal and maximal $L^2$-cohomology groups do not correspond to the cohomology groups of a complex, but they are also referred to as Hodge cohomology groups \cite{HHM2004}, since they can be identified with a subspace of $L^2$-harmonic forms.  Indeed, the weak Kodaira decompositions of the Hilbert complexes \eqref{l2i.6a} and \eqref{l2i.6b} induce natural identifications
\begin{gather}
\label{l2i.9a} L^2_rH^*_{\min}(M;E,g)\cong \cH^*_{\rel}(M;E,g):= \ker \Delta_{\rel}= \ker d_{\min}\cap\ker d^*_{\max}, \\
L^2_rH^*_{\max}(M;E,g)\cong \cH^*_{\abs}(M;E,g):= \ker \Delta_{\abs}= \ker d_{\max}\cap\ker d^*_{\min}. 
\label{l2i.9b}\end{gather}
These identifications show in particular that the dimension of the kernel $\cH^k_{\rel/\max}(M;E,g)$ of $\Delta_{\rel/\abs}$ in degree $k$ only depends on the quasi-isometric class of the metric.  

Following \cite{HM05}, we can associate two other types of Hodge cohomology groups to $(M;E,g)$.  The first one is the \textbf{maximal Hodge cohomology group}, given by
\begin{equation}
 \cH^k_{\max}(M;E,g):= \ker d_{\max,k}\cap \ker d^*_{\max,k-1}\cong \ker d_{\max,k}/ \overline{\Im d_{\min,k-1}},
\label{l2i.10a}\end{equation}
inducing the weak Kodaira decomposition
\begin{equation}
L^2\Omega^k(M;E,g)= \cH^k_{\max}(M;E,g)\oplus \overline{\Im (d_{\min,k-1})}\oplus  \overline{\Im (d^*_{\min,k})}.
\label{l2i.10b}\end{equation}
The other is the \textbf{minimal Hodge cohomology group}, given by
\begin{equation}
\begin{aligned}
 \cH^k_{\min}(M;E,g)&:= \ker d_{\min,k}\cap \ker d^*_{\min,k-1}= \cH^k_{\rel}(M;E,g)\cap \cH^k_{\abs}(M;E,g), \\
  &\cong \ker d_{\min,k}/ \lrp{ \overline{\Im(d_{\max,k-1})}\cap \ker d_{\min,k} }.
 \end{aligned} 
\label{l2i.10a}\end{equation}
For this latter group, there is no weak Kodaira decomposition in general, since as pointed out in \cite{HM05}, the closure of the images of $d_{\max,k-1}$ and $d^*_{\max,k}$ are not orthogonal in general.  However, $\cH^k_{\min}(M;E,g)$ can be realized as the kernel of the Friedrichs extension of the Hodge Laplacian.  Indeed, by the weak Kodaira decomposition \eqref{l2i.10b}, notice that the minimal extension of the Hodge-deRham operator of $(M;E,g)$ is 
$$
  (d+d^*)_{\min}= d_{\min}+d^{*}_{\min} \quad \mbox{with domain} \quad \cD((d+d^*)_{\min})= \cD(d_{\min})\cap \cD(d^*_{\min}),
$$
so that 
$$
        \cH^k_{\min}(M;E,g)= \ker (d+d^*)_{\min}= \ker ((d+d^*)_{\max}(d+d^*)_{\min})= \ker \Delta_{\Fr},
$$
where 
\begin{equation}
\Delta_{\Fr}:=(d+d^*)_{\max}(d+d^*)_{\min}
\label{l2i.11}\end{equation}
is the Friedrichs extension of the Hodge Laplacian.  From the above definitions, it clearly follows that the various Hodge cohomology groups are related via the following diagram of natural inclusions
\begin{equation}
\xymatrix{
      & \cH^k_{\rel}(M;E,g) \ar@{^{(}->}[rd] & \\
     \cH^k_{\min}(M;E,g) \ar@{^{(}->}[ru] \ar@{^{(}->}[rd] & & \cH^k_{\max}(M;E,g) \\
     & \cH^k_{\abs}(M;E,g). \ar@{^{(}->}[ru]&
}
\label{l2i.12}\end{equation}
\begin{lemma}
If the Hodge-deRham operator $\eth=d+d^*$ is essentially self-adjoint, then  
\begin{equation}
 \cH^*_{\min}(M;E,g)=\cH^*_{\rel}(M;E,g)=\cH^*_{\abs}(M;E,g)=\cH^*_{\max}(M;E,g).
\label{l2i.13a}\end{equation}
If furthermore $\eth_{\max}=\eth_{\min}$ is Fredholm, then these spaces are finite dimensional and
\begin{equation}
     L^2H^k_{\min}(M;E,g)=L^2_rH^k_{\min}(M;E,g)\cong  \cH^k_{\min}(M;E,g)
\label{l2i.13b}\end{equation}
with the identifications \eqref{l2i.13a} and \eqref{l2i.13b} valid for any Riemannian metric $g'$ in the quasi-isometric class of $g$.  Moreover, in this case, there is a Poincar\'e duality
\begin{equation}
 (L^2H^k_{\min}(M;E,g))^*\cong L^2H^{\dim M-k}_{\min}(M;E,g) \quad \forall k.
\label{l2i.13c}\end{equation}
\label{l2i.13}\end{lemma}
\begin{proof}
From the diagram \eqref{l2i.12}, it suffices to show that $\cH^*_{\min}(M;E,g)=\cH^*_{\max}(M;E,g)$ to establish \eqref{l2i.13a}.  By assumption, $\eth_{\min}=\eth_{\max}$, so
$$
   \cD(d_{\min})\cap \cD(d^*_{\min})\subset \cD(d_{\max})\cap\cD(d^*_{\max})\subset \cD(\eth_{\max})=\cD(\eth_{\min})=\cD(d_{\min})\cap \cD(d^*_{\min}).
$$
This means that
$$
 \cD(d_{\min})\cap \cD(d^*_{\min})= \cD(d_{\max})\cap\cD(d^*_{\max}),
$$
which implies that 
$$
    \cH^*_{\min}(M;E,g)= \ker d_{\min}\cap\ker d^*_{\min}= \ker d_{\max}\cap\ker d^*_{\max}= \cH^*_{\max}(M;E,g)
$$
as claimed.  If furthermore $\eth_{\min}$ is Fredholm, then $\cH^*_{\min}(M;E,g)=\ker\eth_{\min}$ is finite dimensional and we deduce from \eqref{l2i.10b} and the Fredholmness of $\eth_{\min}$ that
$$
   L^2\Omega^*(M;E,g)= \cH^*_{\min}(M;E,g)\oplus \Im \eth_{\min}= \cH^*_{\min}(M;E,g)\oplus \Im d_{\min}\oplus \Im d^*_{\min}.
$$
In particular, $\Im d_{\min}= \overline{\Im d_{\min}}$, so
$$
 L^2H^*_{\min}(M;E,g)=L^2_rH^*_{\min}(M;E,g).
$$
Since $L^2H^*_{\min}(M;E,g)$, $L^2_rH^*_{\min}(M;E,g)$ and the dimension of the spaces in \eqref{l2i.13a} only depend on the quasi-isometric class of the metric $g$, we see that \eqref{l2i.13a} and \eqref{l2i.13b} also hold for any metric quasi-isometric to $g$.  Since the Hodge star operator induces the Poincar\'e duality
\begin{equation}
      (\cH^{k}_{\rel}(M\setminus\pa M;E,g))^*\cong \cH^{\dim M-k}_{\abs}(M;E,g) \quad \forall k,
\label{l2i.13d}\end{equation}
the Poincar\'e duality \eqref{l2i.13c} follows from \eqref{l2i.13d} and the identifications \eqref{l2i.13a} and \eqref{l2i.13b}.
\end{proof}

Relying on \cite{ALMP2012}, the previous result applies as follows to wedge metrics.

\begin{theorem}
Let $M$ be a compact oriented manifold with fibered corners.  Suppose also that for each $H\in \cM_1(M)$, $S_H$ is also oriented. Let $g_w$ be an associated wedge metric and let $E\to M$ be a flat Euclidean vector bundle on $M$.  If for each $H\in \cM_1(M)$ and $s\in S_H$, 
\begin{equation}
 \cH^{\frac{\dim\phi_H^{-1}(s)}2}_{\min}(\phi^{-1}_H(s)\setminus \pa \phi_H^{-1}(s);E,\kappa_{w,H,s})=\{0\}
\label{l2i.14a}\end{equation}
for $\kappa_{w,H,s}$ a wedge metric on $\phi_H^{-1}(s)$,
then 
\begin{equation}
L^2H^*_{\min}(M\setminus \pa M;E,g_w)=L^2_rH^*_{\min}(M\setminus \pa M;E,g_w) \cong \cH^*_{\min}(M\setminus \pa M;E,g_w) 
\label{l2i.14b}\end{equation} 
and 
\begin{equation}
 \cH^*_{\min}(M\setminus \pa M;E,g_w) =\cH^*_{\rel}(M\setminus \pa M;E,g_w)=\cH^*_{\abs}(M\setminus \pa M;E,g_w)=\cH^*_{\max}(M\setminus \pa M;E,g_w)
\label{l2i.14c}\end{equation}
with all these groups finite dimensional.  Moreover, there is a Poincar\'e duality 
\begin{equation}
(L^2H^k_{\min}(M\setminus \pa M;E,g_w))^*\cong L^2H^{\dim M-k}_{\min}(M\setminus \pa M;E,g_w) \quad \forall k.
\label{l2i.14d}\end{equation}
\label{l2i.14}\end{theorem}
\begin{proof}
The condition \eqref{l2i.14a} ensures that the Witt condition of \cite[(5.4) b)]{ALMP2012} is satisfied.  By \cite[Proposition~5.4 and Theorem~1.1]{ALMP2012}, there exists a wedge metric $\hat{g}_w$ on $M$ such that the associated Hodge-deRham operator is essentially self-adjoint and Fredholm, so the result follows from Lemma~\ref{l2i.13}.  Technically speaking, the results \cite[Proposition~5.4 and Theorem 1.1]{ALMP2012} are formulated with $E$ trivial of rank $1$, but the same results hold with essentially the same proof when the Hodge-deRham operator acts on the sections of a flat Euclidean vector bundle $E\to M$.   
\end{proof}
\begin{remark}
If the fibers of $\phi_H$ are odd dimensional for each $H\in \cM_1(M)$, the condition \eqref{l2i.14a} is trivially satisfied for any flat Euclidean vector bundle $E$.  In the examples we will consider in subsequent sections, this is how we will check that condition \eqref{l2i.14a} holds.
\label{l2i.15}\end{remark}
\begin{remark}
If the associated smoothly stratified space $\hM_{\phi}$ is an orbifold and the wedge metric $g_w$ is smooth in the orbifold sense, then condition \eqref{l2i.14a} is also automatically satisfied since $\phi_H^{-1}(s)$ is the finite quotient of a sphere by a finite subgroup of the orthogonal group.  In this case, \cite[Theorem~1.1]{ALMP2012} applies directly to $g_w$ to show that the corresponding Hodge-deRham operator is essentially self-adjoint and Fredholm.  In fact, by elliptic regularity, its unique self-adjoint extension has domain corresponding to the corresponding orbifold $L^2$-Sobolev space of order 1 \cite{Farsi2001}.  

More precisely, to see that  \cite[Theorem~1.1]{ALMP2012} applies directly to $g_w$, we need for each $H\in\cM_1(M)$ and $s\in S_H$ to check that the Hodge-deRham operator associated to the metric $\kappa_{w,H,s}$ has no eigenvalue in $(-1,1)\setminus\{0\}$ \cite[Assumption (5.4) a)]{ALMP2012}.  But since $(\phi_H^{-1}(s),\kappa_{w,H,s})$ corresponds to a quotient of the unit sphere with its canonical metric by the action of a finite subgroup of orthogonal transformations, the fact that the Hodge-deRham operator   has no eigenvalue in $(-1,1)\setminus\{0\}$ follows from the Gallot-Meyer result \cite{GM1975}.

\label{l2i.15b}\end{remark}

Coming back to a possibly incomplete oriented Riemannian manifold $(M,g)$ equipped with a flat Euclidean vector bundle $E\to M$,  suppose now that $g$ is K\"ahler with complex structure $I$.  Let $E_{\bbC}$ be the complexification of $E$, namely $E_{\bbC}$ is the flat Hermitian vector bundle with fiber above $m$ given by $E_m\otimes_{\bbR}\bbC$.  This vector bundle is automatically holomorphic.  There is also a decomposition
\begin{equation}
L^2\Omega^k(M;E_{\bbC},g)= \bigoplus_{p+q=k} L^2\Omega^{p,q}(M;E_{\bbC},g))
\label {l2i.16}\end{equation}
with
$$
L^2\Omega^{p,q}(M;E_{\bbC},g))= L^2(M;\Lambda^p(T^{1,0}M)^*\wedge \Lambda^q(T^{0,1}M)^{*}\otimes E_{\bbC},g),
$$
where $T^{1,0}M$ and $T^{0,1}M$ are the subbundles of the complexification $T_{\bbC}M$ of the tangent bundle $TM$ on which $I$ acts by multiplication by $\sqrt{-1}$ and $-\sqrt{-1}$ respectively.  There are also natural operators
$$
\db: \Omega^{p,q}_c(M;E_{\bbC})\to \Omega^{p,q+1}_c(M;E_{\bbC}) \quad \mbox{and} \quad \pa: \Omega^{p,q}_c(M;E_{\bbC})\to \Omega^{p+1,q}_c(M;E_{\bbC})
$$
such that the exterior differential decomposes as $d=\pa+\db$.  If $\db^*$ is the formal adjoint of $\db$, then we can consider the associated Dolbeault operator $\db+\db^*$.  It is well-known (see for instance \cite{Huybrechts}) that the corresponding Laplacian
$$
   \Delta_{\db}:= (\db+\db^*)^2= \db\db^*+\db^*\db,
$$
is half the Hodge Laplacian, namely 
\begin{equation} 
\Delta_{\db}=\frac12 (d+d^*)^2.
\label{l2i.17}\end{equation}  
\begin{lemma}
If $g$ is a K\"ahler metric, then the self-adjoint extension
$$
    2(\db+\db^*)_{\max}(\db+\db^*)_{\min}
$$
of $2\Delta_{\db}$ coincides with the Friedrichs extension \eqref{l2i.11} of the Hodge Laplacian acting on forms taking values in $E_{\bbC}$.  
\label{l2i.18}\end{lemma}
\begin{proof}
Given the identification \eqref{l2i.17}, this is a standard argument, see for instance the proof of \cite[Lemma~3.1]{BL93}.
\end{proof}
Since $\Delta_{\db}$ preserves the bidegree of the decomposition \eqref{l2i.16}, this yields the following.
\begin{proposition}
If $g$ is a K\"ahler metric, then
$$
   \cH^k_{\min}(M;E_{\bbC},g)= \bigoplus_{p+q=k} \cH^{p,q}_{\min}(M;E_{\bbC},g),
$$
where 
$$
      \cH^{p,q}_{\min}(M;E_{\bbC},g)=  \cH^{p+q}_{\min}(M;E_{\bbC},g)\cap L^2\Omega^{p,q}(M;E_{\bbC},g).
$$
In particular, the complex structure $I$ acts unitarily on $\cH^k_{\min}(M;E_{\bbC},g)$ and orthogonally on $\cH^k_{\min}(M;E,g)$.
\label{l2i.19}\end{proposition}
\begin{remark}
To show that $I$ acts unitarily on $\cH^k_{\min}(M;E_{\bbC},g)$ (the kernel of the Friedrichs extension of the Hodge Laplacian), we can also proceed as in the proof of \cite[Theorem~5.9]{BL93}.
\label{l2i.19f}\end{remark}
If $\omega$ is the K\"ahler form of $g$, we can consider the bounded operator
$$
    \begin{array}{llcl}
    L: & L^2\Omega^k(M;E_{\bbC},g)& \to & L^2\Omega^{k+2}(M;E_{\bbC},g) \\
        & \eta & \mapsto &\omega\wedge \eta
    \end{array}
$$
and its adjoint $L^*$.  
\begin{corollary}
The operators $L$ and $L^*$ induce well-defined maps
\begin{equation}
 L: \cH^k_{\min}(M;E_{\bbC},g) \to  \cH^{k+2}_{\min}(M;E_{\bbC},g) 
\label{l2i.20a}\end{equation}
and
\begin{equation}
 L^*: \cH^k_{\min}(M;E_{\bbC},g) \to  \cH^{k-2}_{\min}(M;E_{\bbC},g). 
\label{l2i.20b}\end{equation}
\label {l2i.20}\end{corollary}
\begin{proof}
Since the K\"ahler form is a closed $2$-form, 
\begin{equation}
    [L,d]=0.
\label{l2i.21}\end{equation}
Taking the formal adjoint of this equation also yields
$$
  [L^*,d^*]=0.
$$
On the other hand, it is well-known, see for instance \cite[Proposition~3.1.12]{Huybrechts}, that
$$
    [L,d^{*}]= d^c \quad \mbox{and} \quad [L^*,d]= -(d^c)^*,
$$
where 
$$
   d^c= -I^*d I= \sqrt{-1}(\db-\pa).
$$
To check that the map \eqref{l2i.20a} is well-defined, it suffices then to use the fact that $I$ acts unitarily on $\cH^k_{\min}(M;E_{\bbC},g)$.  Indeed, given $\eta\in \cH^{k}_{\min}(M;E_{\bbC},g)$, we know by Proposition~\ref{l2i.19} that $I\eta\in \cH^{k}_{\min}(M;E_{\bbC},g)$, hence
$$
  d_{\min}L\eta= Ld_{\min} \eta=0 \quad \mbox{by \eqref{l2i.21}}
$$
and
$$
d^*_{\min}L\eta=Ld^*_{\min} \eta -[L,d^*_{\min}]\eta=0-d^c\eta= + I^* d_{\min}I\eta =0,
$$
showing that $L\eta\in \cH^{k+2}_{\min}(M;E_{\bbC},g)$.  One can show similarly that the map \eqref{l2i.20b} is well-defined.
\end{proof}
This can be used to deduce the following $L^2$-version of the Hard Lefschetz theorem.
\begin{corollary}
If $\dim \cH^*_{\min}(M;E_{\bbC},g)<\infty$, then the operators $L$ and $L^*$ induce isomorphisms
\begin{equation}
 L^k: \cH^{\frac{\dim M}2-k}_{\min}(M;E_{\bbC},g)\to  \cH^{\frac{\dim M}2+k}_{\min}(M;E_{\bbC},g)
\label{l2i.22a}\end{equation}
and 
\begin{equation}
 (L^*)^k: \cH^{\frac{\dim M}2+k}_{\min}(M;E_{\bbC},g)\to  \cH^{\frac{\dim M}2-k}_{\min}(M;E_{\bbC},g)
\label{l2i.22b}\end{equation}
for $k\in\{1,\ldots, \frac{\dim M}2\}$.
\label{l2i.22}\end{corollary}
\begin{proof}
Given Corollary~\ref{l2i.20}, this is a standard argument relying on the representation of theory of $\mathfrak{sl}(2,\bbC)$ and the fact that 
$$
    [L,L^*]= n-\frac{\dim M}2
$$
on forms of degree $n$, see for instance \cite[Proposition~3.3.13]{Huybrechts}.
\end{proof}

In particular, for wedge metrics, we can combine Theorem~\ref{l2i.14} with Corollary~\ref{l2i.22} to obtain the following.
\begin{corollary}
Let $M$, $g_w$ and $E$ be as in Theorem~\ref{l2i.14}.  If the metric $g_w$ is K\"ahler, then the operator $L$  induces an isomorphism
\begin{equation}
 L^k: L^2H^{\frac{\dim M}2-k}_{\min}(M\setminus \pa M;E_{\bbC},g_w)\to  L^2H^{\frac{\dim M}2+k}_{\min}(M\setminus\pa M;E_{\bbC},g_w)
\label{l2i.23a}\end{equation}
for $k\in\{1,\ldots, \frac{\dim M}2\}$.  Furthermore, defining the primitive $L^2$-cohomology groups of degree $k$ by
$$
   L^2P^m_{\min}(M\setminus\pa M;E_{\bbC},g_w):= \ker L^*\cap \cH^{m}_{\min}(M\setminus\pa M;E_{\bbC},g_w)
$$
yields the Lefschetz decomposition
\begin{equation}
 L^2H^{m}_{\min}(M\setminus\pa M;E_{\bbC},g_w)= \bigoplus_{k} L^2P^{m-2k}_{\min}(M\setminus\pa M;E_{\bbC},g_w).
\label{l2i.23c}\end{equation}

\label{l2i.23}\end{corollary}

\section{Wedge $3$-Sasakian manifolds}\label{bs.0}

In this section, we will briefly review the notion of $3$-Sasakian manifold and allow for possible singularities of wedge type.  In this singular setting, we will then show that $L^2$-harmonic forms are $\Sp(1)$-invariant with respect to the natural $\Sp(1)$-action.   

Consider a Riemannian manifold $(\cS,g)$ with Levi-Civita connection $\nabla$.  For $\xi$ a vector field on $\cS$, let $\eta$ denote the $1$-form dual to $\xi$ and let $\Xi$ be the endomorphism of the tangent bundle defined by $\Xi(X)=\nabla_X\xi$.  Notice that $\xi$ will be a Killing vector field if and only if $\Xi$ is skew-symmetric.

\begin{definition}
The triple $(\cS,g,\xi)$ is a \textbf{Sasakian manifold} if $\xi$ is a Killing vector field of unit length and 
$$
       (\nabla_X\Xi)(Y)= \eta(Y)X-g(X,Y)\xi
$$
for all vector fields $X$ and $Y$.  In this case, we say $g$ is a \textbf{Sasakian metric}.
\label{bs.1}\end{definition}
Referring to \cite{Boyer-Galicki} and references therein for more details, let us recall that one of the main features of a Sasakian manifold is that the associated cone metric 
$$
     dr^2+ r^2g
$$
on $\bbR^+\times\cS$ is K\"ahler.  In particular, Sasakian manifolds are always odd dimensional.  In terms of the complex structure $J$ of the K\"ahler cone, the Killing vector field $\xi$ is then given by $J(r\frac{\pa}{\pa r})$ when $\cS$ is identified with the cross-section $\{1\}\times \cS$ of the cone, while the K\"ahler form of the K\"ahler cone metric is given by
$$
           \omega= \frac{\sqrt{-1}}2 \pa\db r^2.  
$$
When the K\"ahler cone is Ricci-flat, the Sasakian manifold is Einstein with positive Einstein constant equal to $\dim \cS-1$.  Requiring furthermore  that the K\"ahler cone be hyperK\"ahler yields the following structure on $\cS$.

\begin{definition}
A \textbf{$3$-Sasakian manifold} is a Riemannian manifold $(\cS,g)$ such that the cone metric $dr^2+r^2g$ on $\bbR^+\times \cS$ is hyperK\"ahler.  Equivalently, a $3$-Sasakian manifold is a Riemannian manifold $(\cS,g)$ admitting three distinct Sasakian structures with Killing vector fields $\xi^1,\xi^2$ and $\xi^3$ mutually orthogonal and such that

\begin{equation}
[\xi^a,\xi^b]= \sum_{c=1}^3 \epsilon^{abc}\xi^c \quad \mbox{for} \quad a,b,c\in\{1,2,3\}.
\label{bs.2a}\end{equation}

\label{bs.2}\end{definition} 

By \eqref{bs.2a}, the vector fields $\xi^1,\xi^2$ and $\xi^3$ generate a Lie algebra naturally isomorphic to the Lie algebra of $\Sp(1)$.  In fact, by Frobenius theorem, they induce a foliation $\cF$ on $\cS$ and correspond to the infinitesimal generators of an action of $\Sp(1)$ on $\cS$ with the leaves of $\cF$ corresponding to the orbits of this $\Sp(1)$-action.  The $3$-Sasakian structure behaves nicely with respect to this foliation.  More precisely, referring to \cite[Theorem~1.3]{Galicki-Salamon1996} or \cite[Proposition~13.3.11 and Theorem~13.3.13]{Boyer-Galicki} for further details and references, there is the following well-known result.

\begin{theorem}
Let $(\cS,g,\xi^a)$ be a $3$-Sasakian manifold of dimension $4n+3$ such that the vector fields $\xi^1,\xi^2$ and $\xi^3$ are complete.  Then:

\begin{enumerate}
\item $g$ is Einstein with scalar curvature $2(2n+1)(4n+3)$;
\item The foliation $\cF$ generated by $\xi^1,\xi^2$ and $\xi^3$ is Riemannian with respect to the metric $g$;
\item Each leaf is totally geodesic and of constant curvature $1$, while the space of leaves $Q$ is a quaternionic-K\"ahler orbifold of scalar curvature $16n(n+2)$;
\item The natural projection $\pi: \cS\to Q$ is a principal orbibundle with group $\Sp(1)$ or $\SO(3)$.
\end{enumerate}

\label{bs.3}\end{theorem}

In particular, a $3$-Sasakian manifold is Einstein with positive Einstein constant $\dim \cS-1$.  By Myers's theorem, a complete $3$-Sasakian manifold is therefore compact and has finite fundamental group.  This implies that its first Betti number vanishes.  More generally, it was shown by Galicki and Salamon \cite[Theorem~A]{Galicki-Salamon1996} that the odd Betti numbers $b_{2k+1}$ of $\cS$ vanish for $0\le k\le \frac{\dim \cS-3}4$.  

Motivated by the study of quiver varieties, where the hyperK\"ahler cones showing up are typically singular, we want to extend and refine this vanishing result to singular $3$-Sasakian metrics.  We will concentrate our effort on the case where the singular metric is an exact wedge metric.  
\begin{definition}
An exact wedge metric $g_w$ on a manifold with fibered corners $\cS$ is \textbf{$3$-Sasakian} if $(\cS\setminus\pa\cS,g_w)$ is $3$-Sasakian with Killing vector fields $\xi^1,\xi^2$ and $\xi^3$ extending to complete wedge vector fields in $\CI(\cS;{}^{w}T\cS)$  generating an action of $\Sp(1)$ such that for each $H\in \cM_1(\cS)$:
\begin{enumerate}
\item If $a_e: {}^{e}T\cS\to T\cS$ denotes the anchor map of the edge cotangent bundle, then for $a\in\{1,2,3\}$, $a_e(v\xi^a)|_H=0$ and $\xi^a$ descends to a wedge vector field $\xi^a_{S_H}\in \CI(S_H;{}^{w}TS_H)$ on the base $S_H$ of $\phi_H:H\to S_H$ making the exact wedge metric $g_{S_H,w}$ on $S_H$ induced by $g_w$ a $3$-Sasakian metric with Killing vector fields $\xi^1_{S_H},\xi^2_{S_H}$ and $\xi^3_{S_H}$;
\item For each $s\in S_H$, the exact wedge metric $g_{w,s}$ on the manifold with fibered corners $\phi_H^{-1}(s)$ induced by $g_w$ is such that 
$$
 dx_H^2 + x_H^2 g_{w,s}
$$
is a hyperK\"ahler cone making $g_{w,s}$ an exact wedge $3$-Sasakian metric.  
\end{enumerate}
\label{bs.4}\end{definition} 
\begin{remark}
Since the fibers and the base of $\phi_H:H\to S_H$ have depth lower than $\cS$, the definition above is not circular, namely proceeding by induction on the depth of $\cS$, we can assume that the notion of exact wedge $3$-Sasakian  metric is well-defined on manifolds with fibered corners of lower depth.
\end{remark}
Models at infinity of quiver varieties yield natural examples of exact wedge $3$-Sasakian metrics.
\begin{example}
Let $\mathfrak{M}_{\zeta}$ be a quiver variety as in Corollary~\ref{cqac.21}.  Then the model wedge exact metric $g_w$ in \eqref{pt.1} for the maximal hypersurface of the $\QAC$ compactification of $\mathfrak{M}_{\zeta}$ is an exact wedge $3$-Sasakian metric. Indeed, by Corollary~\ref{cqac.21}, the cone metric $dr^2+r^2g_w$ is hyperK\"ahler since it corresponds to the singular quiver variety $\mathfrak{M}_{0}$.  By the construction leading to Theorem~\ref{cqac.14} and Corollary~\ref{cqac.21}, condition (2) of Definition~\ref{bs.4} holds with the hyperK\"ahler cone $\hat{\mu}^{-1}_m(0)/G_m$ for $m\in \Sigma_{H_i}$ representing a point in the base $\Sigma_{H_i}/G$ of a boundary hypersurface $H_i\cap\widetilde{\mu^{-1}(\zeta)}/G$ of $\widetilde{\mu^{-1}(\zeta)}/G$.  On the other hand, condition (1) in Definition~\ref{bs.4} follows from a result of Dancer-Swann \cite{Dancer-Swann} (see also \cite[Theorem~1.1]{Mayrand}) applied to the hyperK\"ahler cone $\mathfrak{M}_0$.  This shows at the same time that exact wedge metrics of \eqref{pt.1} for the other boundary hypersurfaces of the $\QAC$ compactification of $\mathfrak{M}_{\zeta}$ are also exact wedge $3$-Sasakian manifolds.
\label{ew3S}\end{example}

Theorem~\ref{bs.3} naturally extends to exact wedge $3$-Sasakian  metrics.  Indeed, by Definition~\ref{bs.4}, the vector fields $\xi^1$, $\xi^2$ and $\xi^3$ are complete wedge vector fields on $\cS$, \ie they are also smooth vector fields on $\cS$ and their flows exist for all time on $\cS$, so generate a locally free action of $\Sp(1)$ on $\cS$ such that for each $H\in \cM_1(\cS)$, there is a corresponding action on $S_H$ making the map $\phi_H: H\to S_H$ $\Sp(1)$-equivariant.    The $\Sp(1)$-action also descends to a $\Sp(1)$-action on the smoothly stratified space $\hat{\cS}$ associated to $\cS$.   Since the action is locally free and $\Sp(1)$ is compact, the foliation induced by the orbits of this action is automatically quasi-regular, that is, the quotient $\cS/\Sp(1)$ is a orbifold with corners in the sense of \cite{CDR}.  Notice that
 $\hat{Q}= \hat{\cS}/\Sp(1)$ is naturally a smoothly stratified space with strata coming from those of $\cS$ and orbifold singularities created by taking the quotient of the $\Sp(1)$-action.  More precisely, by a result of Tanno \cite{Tanno}, see also \cite[Proposition~13.3.11]{Boyer-Galicki}, if the $\Sp(1)$ action is nowhere free, then the smallest conjugacy class of the stabilizer groups is the one of $\bbZ_2\subset \Sp(1)$.  Thus, setting 
$$
          G= \left\{ \begin{array}{ll} \SO(3)=\Sp(1)/\bbZ_2, & \mbox{the $\Sp(1)$-action is nowhere free}, \\
              \Sp(1), & \mbox{otherwise},  \end{array} \right.
$$
the stratification on $\cS$ induced by the conjugacy classes of stabilizer subgroups of the $\Sp(1)$-action is given by strata of the form
$$
    s_I= \{p\in \cS \; | \; G_p\in I\}
$$
for I a conjugacy class of subgroups in $G$ and $G_p\subset G$ the stabilizer group of $p$ in $G$.   Let $\overline{s}_{I}$ be the closure $s_I$ in $\cS$.  To resolve the $G$-action  on $\cS$ into a free action, we could as in \cite{AM2011} blow up the closed strata $\overline{s}_I$ in an order compatible with the partial order on the strata given by
$$
\begin{aligned}
       s_I<s_J \; &\Longleftrightarrow \; s_I \subsetneq \overline{s}_J \\
                        & \Longleftrightarrow \; \mbox{for $K\in J$, there exists $L\in I$ such that $K\subset L$.} 
\end{aligned}       
$$
However, to describe $\hat{\cS}/\Sp(1)=\hat{\cS}/G$ as a smoothly stratified space, each stratum $s_I$ needs to be decomposed as a disjoint union
$$
     s_I= s_{I,\cS} \sqcup \left( \bigsqcup_{H\in \cM_1(\cS)} s_{I,H} \right),
$$
where 
$$
  s_{I,H}= s_I\cap \left( H\setminus \left(\bigcup_{L<H} L\cap H \right) \right)
$$
for $H\in \cM_1(\cS)$ and where $s_{I,\cS}= s_I\setminus (s_I\cap \pa \cS)$.  Proceeding lexicographically, there is a partial order on this refined decomposition given by
\begin{equation}
  s_{I,H}<s_{J,L} \; \Longleftrightarrow \; I<J, \quad \mbox{or $I=J$ and $H<L$}, 
\label{bs.5}\end{equation}
where we used the convention that $H<\cS$ for all $H\in \cM_1(\cS)$ when $L=\cS$.  If $\overline{s}_{I,H}$ denotes the closure of $s_{I,H}$, then we can resolve the $G$-action into a free action on the space
\begin{equation}
 Y:= [\cS; \{\overline{s}_{I,H}\}, I\in\cI\setminus\{I_{\Id}\}, \; H\in \cM_1(\cS)\cup \{\cS\}],
\label{bs.6}\end{equation}
obtained from $\cS$ by blowing up the $\overline{s}_{I,H}$ in an order compatible with the partial order described above,
where $\cI$ is the set of conjugacy classes of subgroups of $G$ and $I_{\Id}$ is the conjugacy class corresponding to the trivial subgroup $\{\Id\}$.  One can readily check that the quotient 
$$
        Q=Y/G
$$
is naturally a manifold with fibered corners with associated smoothly stratified space $\hat{Q}=\hat{\cS}/G$.  This yields the following generalization of Theorem~\ref{bs.3}.

\begin{theorem}
Let $g_{w}$ be an exact wedge $3$-Sasakian  metric on a manifold with fibered corners $\cS$ of dimension $4n+3$.  Then:
\begin{enumerate}
\item $g_w$ is Einstein with scalar curvature $2(2n+1)(4n+3)$;
\item The foliation $\cF$ generated by $\xi^1,\xi^2$ and $\xi^3$ is Riemannian with respect to the metric $g_w$ on $\cS\setminus \pa \cS$ and with respect to $g_{S_H,w}$ on $S_{H\setminus}\pa S_H$ for each $H\in \cM_1(\cS)$;
\item Each leaf is totally geodesic and of constant curvature $1$, while on the quotient $Q=Y/ G$, the metric $g_w$ induces a quaternionic-K\"ahler exact wedge metric of scalar curvature $16n(n+2)$;
\item The natural projection $\pi: \cS\to \cS/\Sp(1)$ is a principal orbibundle with group $\Sp(1)$ or $\SO(3)$.  
\end{enumerate}
\label{bs.7}\end{theorem}

Since the fibers of the fiber bundles of the iterated fibration structure of $\cS$ admit exact wedge $3$-Sasakian metrics, they are odd dimensional.  By Remark~\ref{l2i.15}, this means that Theorem~\ref{l2i.14} holds for $g_w$ on $\cS$ for any flat Euclidean vector bundle $E$.  This is also the case when $g_w$ is seen as a wedge metric on $Y\setminus \pa Y$ as the next lemma shows.
  \begin{lemma}
  Let $Y$ be the manifold with fibered corners of \eqref{bs.6}.  Then the fibers of the fiber bundles of the iterated fibration structure of $Y$ are all odd dimensional.
  \label{inv.1}\end{lemma}
 \begin{proof}
 For $H\in\cM_1(Y)$ corresponding to the lift of a boundary hypersurface of $\cS$ to $Y$, the dimension of the fibers of the associated fiber bundle is odd since they admit an exact wedge $3$-Sasakian metric.  For $H\in \cM_1(Y)$ coming from the blow-up of $s_{I,\cS}$, notice that $s_{I,\cS}$ is of dimension $4k+3$ for some $k$ since the corresponding stratum on the quotient $\cS/\Sp(1)$ is of codimension $4$, the orbifold singularities being compatible with the quaternionic-K\"ahler structure.  Hence, since $\dim H$ is even, the dimension of the fibers of the associated fiber bundle must be odd.  Finally, if $H\in\cM_1(Y)$ is a boundary hypersurface associated to $s_{I,H'}$ for some $H'\in \cM_1(S)$ and $I$ a conjugacy class of subgroups of $G$, then the dimension of the fibers of the associated fiber bundle is $f_{H'}+f_I+1$, where $f_{H'}$ is the dimension of the fibers of the fiber bundle associated to $H'\in \cM_1(S)$ and $f_I$ is the dimension of the fibers of the fiber bundle associated to the boundary hypersurface associated to $s_{I,\cS}$.  By the discussion above, $f_{H'}$ and $f_I$ are odd, so $f_{H'}+f_I+1$ is odd as well.
 \end{proof} 
  
We will need to work with the metric $g_w$ both as a wedge metric on $\cS$ and $Y$ and it will be important that the various Hodge cohomology spaces are the same.
\begin{proposition}
Let $g_w$ be an exact wedge $3$-Sasakian metric as in Theorem~\ref{bs.7}. Then for any flat Euclidean vector bundle $E\to \cS$, the conclusions of Theorem~\ref{l2i.14} hold for $g_w$ seen as a wedge metric on $\cS$ or $Y$.  Moreover, the Hodge cohomology groups in \eqref{l2i.14c} are the same whether $g_w$ is seen as a wedge metric on $\cS$ or $Y$. 
\label{inv.2}\end{proposition}
\begin{proof}
The first assertion follows from the discussion above and Remark~\ref{l2i.15}.  For the second assertion, let $\hat{g}_w$ be a wedge metric on $Y$ such that the associated Hodge-deRham operator $\hat{\eth}_w$ is essentially self-adjoint and Fredholm.  Recall from \cite[Proposition~5.4 and Theorem~1.1]{ALMP2012} that such a metric can be obtained from $g_w$ by scaling the wedge metrics in the fibers of the fiber bundles of the iterated fibration structure to ensure that the corresponding Hodge-deRham operator has no eigenvalue in $(-1,1)\setminus\{0\}$ \cite[Assumption (5.4) a)]{ALMP2012}.  If $H_1,\ldots,H_\ell$ is an exhaustive list of the boundary hypersurfaces of $Y$ compatible with the partial order coming from the iterated fibration structure, then one has to scale the metrics in the fibers of $H_{\ell}$ ,then those in the fibers of $H_{\ell-1}$ and so on until we reach $H_1$ to scale the metrics in its fibers.  However, since $g_w$ is smooth on the interior of $\cS$, the fibers of $H_i$ for $H_i\in\cM_{1}(Y)$ associated to the blow-up of $s_{I,\cS}$ are spheres of dimension at least $3$ (possibly blown-up at some submanifolds) with $g_w$ inducing on such a fiber the standard round metric.   Hence, as in Remark~\ref{l2i.15b}, by the Gallot-Meyer result \cite{GM1975}, there is no need to scale the metrics in this case, so this means we may only need to scale the fiber metrics for boundary hypersurfaces corresponding to the lift of a boundary hypersurface of $\cS$ or coming from the blow-up of $s_{I,H}$ for some boundary hypersurface $H\in\cM_1(\cS)$.  

Hence, without loss of generality, we can assume that $\hat{g}_w$ is smooth on $\cS\setminus \pa \cS$.  Let $(\hat{\eth}_w)^{\cS}_{\min}$ and $(\hat{\eth}_w)^{\cS}_{\max}$ be the minimal and maximal extensions of $\hat{\eth}_w$ seen as an operator on $\cS\setminus \pa \cS$.  Similarly, let $(\hat{\eth}_w)^{Y}_{\min}$ and $(\hat{\eth}_w)^{Y}_{\max}$ be the minimal and maximal extensions of $\hat{\eth}_w$ seen as an operator on $Y\setminus \pa Y$.  From the definition of the minimal and maximal extensions, we have the sequence of inclusions
$$
  \cD((\hat{\eth}_w)^{Y}_{\min}) \subset \cD((\hat{\eth}_w)^{\cS}_{\min}) \subset \cD((\hat{\eth}_w)^{\cS}_{\max}) \subset \cD((\hat{\eth}_w)^{Y}_{\max}).  
$$
Since $\hat{\eth}_w$ is essentially self-adjoint on $Y\setminus \pa Y$, this means that all these domains are equal and
$$
     \cH^*_{\min}(Y\setminus \pa Y;E,\hat{g}_w)= \cH^*_{\min}(\cS\setminus \pa \cS;E,\hat{g}_w)= \cH^*_{\max}(\cS\setminus \pa\cS;E,\hat{g}_w)=\cH^*_{\max}(Y\setminus \pa Y;E,\hat{g}_w).
$$
Since the dimension of these spaces only depends on the quasi-isometric class of the metric and since 
$$
\cH^*_{\max}(\cS\setminus \pa \cS;E,g_w)\subset \cH^*_{\max}(Y\setminus \pa Y;E,g_w),
$$
this implies that
$$
 \cH^*_{\max}(\cS\setminus \pa \cS;E,g_w)= \cH^*_{\max}(Y\setminus \pa Y;E,g_w).
$$
The result then follows from this identification and the conclusion of Theorem~\ref{l2i.14} for $g_w$ seen as a metric on $Y$ and $\cS$.
\end{proof}

The previous result allows us to work with $Y$ to draw conclusions on the Hodge cohomology of the exact wedge $3$-Sasakian metric $g_w$ on $\cS$.  On $Y$, the advantage is that the action of $G$ is free and induces a principal $G$-bundle $\pi:Y\to Q$.  On next goal is to show that $L^2$-harmonic forms are invariant with respect to this $G$-action.  For this assertion to make sense, we need to assume that the flat Euclidean vector bundle $E\to \cS$ lifted to $Y$ admits an $\Sp(1)$-action preserving the Euclidean and flat structures and making the bundle projection $E\to Y$ $\Sp(1)$-equivariant.  We will require slightly more.
\begin{definition}
The flat Euclidean vector bundle $E\to \cS$ is \textbf{$\Sp(1)$-equivariant} if it admits an $\Sp(1)$-action preserving the Euclidean and flat structures and making the bundle projection $E\to \cS$ $\Sp(1)$-equivariant.   Furthermore, denoting its lift to $Y$ also by $E$, we say that it is \textbf{nicely $\Sp(1)$-equivariant} if for all $y\in Y$, 
$$
     E|_{\pi^{-1}(y)}\cong \Sp(1)\times_\Gamma \widetilde{E}_y \to \pi^{-1}(y)\cong G= \Sp(1)/\Gamma
$$ 
for some orthogonal representation $\widetilde{E}_y$ of $\Gamma$ with action of $\Sp(1)$ on $E|_{\pi^{-1}(y)}$ given by composition on the left in the first factor in $ \Sp(1)\times_\Gamma \widetilde{E}_y$, where $\Gamma=\bbZ_2$ if the $\Sp(1)$-action on $\cS$ is nowhere free and $\Gamma=\{\Id\}$ otherwise.  
\label{inv.3}\end{definition} 

If $E$ is nicely $\Sp(1)$-equivariant, let $E_{Q,y}$ be the subspace of $\widetilde{E}_y$ fixed by $\Gamma$.  This space coincides with the space of global flat sections of $E|_{\pi^{-1}(y)}$, which corresponds to the kernel of the Laplacian on $\pi^{-1}(y)$ acting on sections of $E|_{\pi^{-1}(y)}$.  As such, as $y$ varies, the subspaces $E_{Q,y}$ combine to form a flat vector bundle $E_Q$ on $Q$.  Of couse, if $E$ is nicely $\Sp(1)$-equivariant, the group $\Sp(1)$ acts on $L^2\Omega^*(Y\setminus \pa Y;E,g_w)$ and this action commutes with the Hodge Laplacian and the Hodge-deRham operator.  There is also an induced action on the minimal Hodge cohomology groups.
\begin{lemma}
Let $g_w$ be an exact wedge $3$-Sasakian metric as in Theorem~\ref{bs.7}.  If $E\to \cS$ is a nicely $\Sp(1)$-equivariant flat Euclidean vector bundle, then each harmonic form in $\cH^*_{\min}(Y\setminus\pa Y;E,g_w)$ is fixed by the action of $\Sp(1)$.   
\label{inv.4}\end{lemma}
\begin{proof}
Given $\Theta\in \mathfrak{sp}(1)$, let $\Theta_*$ be the vector field on $Y\setminus \pa Y$ corresponding to the infinitesimal action of $\Theta$.  Since $E$ is nicely $\Sp(1)$-equivariant, $E$ is locally spanned by flat orthogonal sections that are fixed by the infinitesimal action of $\Sp(1)$.  This means that the Cartan formula
\begin{equation}
   \cL_{\Theta_*}\nu = d\iota_{\Theta_*}\nu + \iota_{\Theta_*}d\nu
\label{hi.1}\end{equation}
holds for $\nu\in\Omega^*(Y\setminus\pa Y;E)$.  Now, by Theorem~\ref{l2i.14} and the identification \eqref{l2i.14b}, the result will follow provided we can show that the natural action of $\Sp(1)$ on $L^2H^*_{\min}(Y\setminus\pa Y;E,g_w)$ is trivial.  Thus, let
$$
  \nu \in \cH^k_{\min}(Y\setminus\pa Y;E,g_w)\cong L^2H^k_{\min}(Y\setminus\pa Y;E,g_w) 
$$
be given.  Since $\Sp(1)$ is connected, given $\Theta\in \mathfrak{sp}(1)$, we need to show that the flow $\Phi_t$ of $\Theta_*$ at time $t=1$ fixes the minimal $L^2$-cohomology class of $\nu$.  Now, we compute that
\begin{equation}
\begin{aligned}
\Phi_1^*\nu- \nu &= \int_0^1 (\frac{d}{dt}\Phi^*_t\nu) dt= \int_0^1 \Phi_t^*(\cL_{\Theta_*}\nu)dt \\
    &= \int_0^1 \Phi_t^*(d\iota_{\Theta_*}\nu) dt, \quad \mbox{by \eqref{hi.1} and the fact $d\nu=0$,} \\
    &= d \int_0^1\Phi_t^*(\iota_{\Theta_*}\nu) dt=du, \quad \mbox{with} \; u:= \int_0^1\Phi_t^*(\iota_{\Theta_*}\nu) dt.
\end{aligned}
\label{hi.2}\end{equation}
Clearly, $u\in L^2\Omega^{k-1}(Y\setminus\pa Y;E,g_w)$.  On the other hand, since $\nu\in \cH^{k}_{\min}(Y\setminus\pa Y;E,g_w)$, there exists a sequence $\{\nu_j\}\subset \Omega_c^k(Y\setminus\pa Y;E)$ such that $\nu_j\to \nu$ and $d\nu_j\to d\nu$ in $L^2$.  If we set 
$$
    u_j:= \int_0^1 \Phi_t^*(\iota_{\Theta_*}\nu_j) dt,
$$
then $u_j\to u$ in $L^2$, while proceeding as in \eqref{hi.2}, we find that
$$
    du_j= (\Phi_1^*\nu_j-\nu_j)\to (\Phi_1^*\nu-\nu)=du \quad \mbox{in $L^2$}.
$$
This shows that $u\in \cD(d_{\min,k-1})$ and that $\Phi_1^*\nu$ represents in $L^2H^k_{\min}(Y\setminus\pa Y;E,g_w)$ the same cohomology class as $\nu$, that is, $\Phi_1$ fixes the minimal $L^2$-cohomology class defined by $\nu$.  
\end{proof}

\section{A vanishing in $L^2$-cohomology for wedge $3$-Sasakian metrics} \label{van.0}

In \cite{Galicki-Salamon1996}, Galicki and Salamon showed that certain cohomology groups are automatically trivial on a closed $3$-Sasakian manifold.  The goal of this section is to generalize this result to the Hodge cohomology groups of an exact wedge $3$-Sasakian metric.      We will follow essentially the same overall strategy as the one of \cite{Galicki-Salamon1996}.  We will need in particular to adapt to incomplete metrics a result of Tachibana \cite{Tachibana} stipulating that harmonic forms below middle degree on a closed Sasakian manifold are horizontal with respect to the orbits of the Reeb vector field.  This will be the occasion to give a `modern' proof of this result.  

Thus, let $g_w$ be an exact wedge $3$-Sasakian metric as in Theorem~\ref{bs.7}.  Fix $a\in \{1,2,3\}$ and set $\xi=\xi^a$.  Then $\xi$ induces a free circle action on $Y$ inducing a circle bundle 
\begin{equation}
 \nu: Y\to B
\label{van.1}\end{equation}
with $B=Y/S^1$ the quotient of this circle action.  As in Theorem~\ref{bs.7} for the quotient of the $\Sp(1)$-action, the base $B$ is naturally a manifold with fibered corners and the metric $g_w$ induces an exact wedge metric $g_B$ on $B$ making 
$$
   \nu: (Y\setminus\pa Y)\to B\setminus \pa B
$$
a Riemannian submersion.  Since the orbits of the $\Sp(1)$-action on $Y$ are never tangent to the fibers of the various fiber bundles of the iterated fibration structure of $Y$, we see by Lemma~\ref{inv.1} that the fibers of the fiber bundles of the iterated fibration structure of $B$ are all odd dimensional.  Hence, by Remark~\ref{l2i.15}, the conclusions of Theorem~\ref{l2i.14} hold for the metric $g_B$ for any flat Euclidean vector bundle on $B$.  On the other hand, by a standard result in Sasakian geometry, the metric $g_B$ is K\"ahler with complex structure induced by the endomorphism $\Xi$ in Definition~\ref{bs.1} and with K\"ahler form $d\eta$, where $\eta$ is the $1$-form dual to $\xi$. In particular, Corollary~\ref{l2i.23}  applies to the metric $g_B$.  

Let $E$ be a flat Euclidean vector bundle on $\cS$ which is nicely $\Sp(1)$-equivariant.  As for the bundle $\pi:Y\to Q$, there is a flat Euclidean vector bundle $E_B\to B$ with fiber $E_{B,b}$ above $b\in B$ corresponding to the global flat sections of $E|_{\nu^{-1}(b)}$ on $\nu^{-1}(b)$.  Let us denote by $d_B$ the exterior differential associated to $E_B$ on $B\setminus \pa B$ and denote by $d_{B,\min}$ its minimal extension with respect to the exact wedge metric $g_B$ and the bundle metric of $E_B$.  Similarly, denote by $d^*_{B,\min}$ the minimal extension of its formal adjoint.
\begin{lemma}
An element $u\in \cH^k_{\min}(Y\setminus\pa Y;E,g_w)$ takes the form
\begin{equation}
    u=\nu^*u_0+ \eta\wedge \nu^*u_1,
\label{van.2a}\end{equation}
where $u_0\in \cD(d_{B,\min,k})\cap\cD(d_{B,\min,k-1}^*)$ is such that $d_{B,\min,k-1}^*u_0=0$ and  $u_1\in \cD(d_{B,\min,k-1})\cap\cD(d^*_{B,\min,k-2})$ is such that $d_{B,\min,k-1}u_1=0$.  
\label{van.2}\end{lemma}  
\begin{proof}
By Lemma~\ref{inv.4}, the form $u$ is $\Sp(1)$-invariant, so in particular $\bbS^1$-invariant with respect to the $\bbS^1$-action generated by the Reeb vector field $\xi$.  Since the $1$-form $\eta$ is also $\bbS^1$-invariant, this means that $u$ is of the form \eqref{van.2a} with $u_i\in L^2\Omega^{k-i}(B\setminus\pa B;E_B,g_B)$.  Since $du=0$, we see that
$$
 0=du= \nu^* du_0+ d\eta\wedge \nu^*u_1-\eta\wedge \nu^*(du_1).
$$
Decomposing in terms of vertical and horizontal degrees with respect to the fiber bundle \eqref{van.1}, this  implies that
\begin{equation}
   du_1=0 \quad \mbox{and} \quad \nu^*du_0+ d\eta\wedge \nu^*u_1=0.
\label{van.3}\end{equation}
In particular, $u_1$ is a closed form.  Since $u\in \cD((d+d^*)_{\min})=\cD(d_{\min})\cap\cD(d^*_{\min})$, there is a sequence $\{v^j\}$ in $\Omega^k_c(Y\setminus\pa Y;E)$ such that $v^j\to u$ and $dv^j\to 0$ in $L^2$ as $j\to \infty$.  Averaging with respect to the $\bbS^1$-action,  we can in fact assume that the terms of the sequence $\{v^j\}$  are $\bbS^1$-invariant, in which case they must be of the form
$$
  v^j= \nu^*v^j_0 + \eta\wedge \nu^*v^j_1
$$
for sequences $\{v^j_i\}\subset \Omega^{k-i}_{c}(B\setminus \pa B;E_B)$.  Since $v^j\to u$ in $L^2$, we must have that $v^j_i\to u_i$ in $L^2\Omega^{k-i}(B\setminus\pa B;E_B;g_B)$ for $i\in\{0,1\}$.  Since 
$$
  dv^j= \nu^*dv^j_0+ d\eta\wedge \nu^*v^j_1 -\eta\wedge \nu^*dv^j_1,
$$ 
we deduce from the fact that $dv^j\to 0$ in $L^2$ that
$$
 dv^j_{1}\to 0 \quad \mbox{and} \quad dv^j_0\to -d\eta\wedge u_1= du_0 \quad \mbox{in $L^2$},
$$ 
showing  that $u_i\in \cD(d_{B,\min,k-i})$ for $i\in\{0,1\}$ as claimed.  Similarly, from $d^*u=0$, we deduce that $d^*u_0=0$ and  $u_i\in \cD(d^*_{B,\min,k-1-i})$ for $i\in\{0,1\}.$ 
\end{proof}
This yields the following singular version of the theorem of Tachibana \cite[Theorem~7.1]{Tachibana}. 
\begin{theorem}
Let $g_w$ be an exact wedge $3$-Sasakian metric on a manifold with fibered corners $\cS$ of dimension $4n+3$ as in Theorem~\ref{bs.7}.  Let $E\to \cS$ be a nicely $\Sp(1)$-equivariant flat Euclidean vector bundle on $\cS$.   For $a\in\{1,2,3\}$ fixed let $\nu: Y\to B$ be the circle bundle generated by the vector field $\xi=\xi^a$.  Then for  $k\le 2n+1$, 
the pull-back by $\nu$ induces an isomorphism
\begin{equation}
\nu^*: \cH^k_{\min}(B\setminus\pa B;E_B,g_B)\cap \ker L^*\to \cH^k_{\min}(Y\setminus\pa Y; E, g_w) \quad \forall \ k\le 2n+1,
\label{van.4b}\end{equation}
where $L^*$ is the operator of Corollary~\ref{l2i.20}.  
\label{van.4}\end{theorem}
\begin{proof}
Let $u\in \cH^k_{\min}(Y\setminus\pa Y; E, g_w)$ be given. By Lemma~\ref{van.2}, $u$ is of the form \eqref{van.2a}.  Let us first show that the cohomology class of $u_1$ vanishes.  If $L$ is the operator defined by
$$
    Lv= d\eta\wedge v
$$
for forms on $B\setminus \pa B$, then by \eqref{van.3},
$$
      Lu_1= -du_0.
$$
By Corollary~\ref{l2i.23} and the fact $u_1$ is of degree $k-1\le 2n+1-1=2n<\frac{\dim B}2$, this means that $u_1$ defines a trivial cohomology class in $L^2H^{k-1}_{\min}(B\setminus\pa B;(E_B)_{\bbC},g_B)$.  This means there exists $v\in \cD(d_{B,\min,k-2})$ such that
$$
  u_1= dv.
$$
But then, the cohomology class represented by $u$ in $L^2H^k_{\min}(Y\setminus\pa Y,g_w,E)$ is also represented by the basic form
$$
   w:=u+ d(\eta\wedge \nu^*v)= \nu^*u_0+ d\eta\wedge \nu^*v.  
$$
This basic form defines a cohomology class in $L^2H^k_{\min}(B\setminus\pa B;E_B,g_B)$ depending on the choice of $v$.  Indeed, adding to $v$ a closed form representing a cohomology class $\psi\in  L^2H^{k-2}_{\min}(B\setminus\pa B; E_{B},g_B)$ changes the cohomology class of $w$ by adding $L\psi$.   In fact, changing $v$ if necessary we can suppose that the cohomology class of $w$ is primitive in terms of the Lefschetz decomposition \eqref{l2i.23c}.  Indeed, if 
$$
   w= w_0+ Lw_2
$$
for closed forms $w_i\in \cD(d_{B,\min,k-2i})$ with $w_0$ representing a primitive cohomology class in $L^2H^{k-2}_{\min}(B\setminus\pa B; E_{B},g_B)$, then replacing $v$ by $v-w_2$ yields the basic form 
$$
     w-d\eta\wedge w_2= w_0.  
$$
Thus, let us choose $v$ so that $w$ defines a primitive cohomology class in $L^2H^k_{\min}(B\setminus\pa B;E_B,g_B)$.  Then its harmonic representative $\hat{w}\in \cH^k_{\min}(B\setminus \pa B;E_B,g_B)$ is such that
$$
   L^*\hat{w}=0.
$$
This ensures that its lift $\nu^*\hat{w}$ to $Y\setminus \pa Y$ is also harmonic, since  $d\nu^*\hat{w}= \nu^*(d\hat{w})=0$ and using the convention that $\eta\wedge (d\eta)^{2n+2}$ is the volume form of $Y\setminus\pa Y$,
$$
\begin{aligned}
       d^*(\nu^*\hat{w})&= -*d*(\nu^*(\hat{w}))= -(-1)^k*d(\eta\wedge (\nu^*(*_B\hat{w})))= -(-1)^k*((d\eta)\wedge \nu^*(*_B\hat{w}) -\eta\wedge \nu^*(d*_B\hat{w})) \\
       &= -(-1)^k*((d\eta)\wedge \nu^*(*_Bw)), \quad \mbox{since $\hat{w}$ is harmonic}, \\
       &=  -\eta\wedge \nu^*(*_B(d\eta\wedge *_B\hat{w}))= (-1)^{k+1}\eta\wedge (\nu^*(L^*\hat{w})) \\
       &= 0, \quad \mbox{since $\hat{w}$ is primitive}.
\end{aligned}       
$$
In particular, this argument shows that the map \eqref{van.4b} is well-defined and clearly injective.
Now, since $\nu^{*}\hat{w}$ and $u$ in $\cH^k_{\min}(Y\setminus\pa Y;E,g_w)$ are two harmonic forms representing the same cohomology class in $L^2H^k_{\min}(Y\setminus\pa Y;E,g_w)$, they must in fact be equal by Theorem~\ref{l2i.14}, showing that the map \eqref{van.4b} is also surjective.  

\end{proof}

\begin{remark}
The proof of Theorem~\ref{van.4} can also be adapted to give a new proof of the original result of Tachibana for closed Sasakian manifolds with $E$ trivial, even in the irregular case.  Indeed, it suffices to replace $L^2H^*_{\min}(B\setminus\pa B;E_B,g_B)$ by the basic cohomology ring of the foliation generated by the Reeb vector field $\xi$ and use the transverse Hodge theorem \cite{EKA1986,EKA1990} (see also \cite[\S~7.2]{Boyer-Galicki}) and the transverse Hard Lefschetz theorem of \cite[\S~3.4.7]{EKA1990} (see also \cite[Theorem~7.2.9]{Boyer-Galicki}).
\label{van.5}\end{remark}

Since the endomorphism $\Xi$ of Definition~\ref{bs.1} corresponds to the horizontal lift of the complex structure on $B\setminus \pa B$, we can deduce the following result from Theorem~\ref{van.4} and Proposition~\ref{l2i.19}.
\begin{corollary}
For $k\le 2n+1$, the endomorphism $\Xi$ induces a well-defined map
$$
   \Xi: \cH^k_{\min}(Y\setminus\pa Y;E,g_w)\to \cH^k_{\min}(Y\setminus\pa Y;E,g_w)
$$
defined by
$$
      (\Xi u)(X_1,\ldots,X_k)= u(\Xi X_1,\ldots,\Xi X_k).
$$
\label{van.6}\end{corollary}

We can also consider the Tachibana operator $T_{\Xi}$ on $k$-forms given by
$$
 ( T_{\Xi}u)(X_1,\ldots,X_k)= \sum_{i=1}^k u(X_1,\ldots,X_{i-1},\Xi X_i,X_{i+1},\ldots,X_k).
$$
Using Theorem~\ref{van.4} and Proposition~\ref{l2i.19}, we obtain the following non-compact version of \cite[Theorem~8.1]{Tachibana}.
\begin{corollary}
For $k\le 2n+1$, the Tachibana operator induces a well-defined map
$$
T_{ \Xi}: \cH^k_{\min}(Y\setminus\pa Y;E,g_w)\to \cH^k_{\min}(Y\setminus\pa Y;E,g_w).
$$
\label{van.7}\end{corollary}
\begin{proof}
It suffices to notice that for a form of pure bidegree  $(p,q)$ in the Hodge decomposition of Proposition~\ref{l2i.19}, the Tachibana operator $T_{\Xi}$ acts by multiplication by $\sqrt{-1}(p-q)$. 
\end{proof}
\begin{remark}
This proof can be adapted to give a different proof of the original result of Tachibana \cite[Theorem~8.1]{Tachibana}.  It suffices to replace Proposition~\ref{l2i.19} by the transverse Hodge decomposition of \cite[Théorème~3.3.3]{EKA1990} (see also \cite[Theorem~7.2.6]{Boyer-Galicki}).
\label{van.7b}\end{remark}

Since $g_w$ is an exact wedge $3$-Sasakian metric, we can apply the previous results with $\xi\in \{\xi^1,\xi^2,\xi^3\}$.  In particular, if we let $\Xi^a$ denote the endomorphism associated to $\xi^a$, then by Corollary~\ref{van.6}, it induces a natural map
\begin{equation}
   \Xi^a: \cH^k_{\min}(Y\setminus\pa Y;E,g_w)\to \cH^k_{\min}(Y\setminus\pa Y;E,g_w)
\label{van.8}\end{equation}
for $k\le 2n+1$.  Since by \cite[(13)]{Galicki-Salamon1996},  the endomorphisms $\Xi^1,\Xi^2$ and $\Xi^3$ satisfy the  relations
\begin{equation}
    \Xi^a\circ\Xi^b= (-\delta^{ab})^k\Id+ \sum_c(\epsilon^{abc})^k\Xi^c
\label{van.9}\end{equation}
when acting on $\cH^k_{\min}(Y\setminus\pa Y;E,g_w)$ for $k\le 2n+1$, 
this yields the following generalization of the vanishing theorem of Galicki and Salamon \cite{Galicki-Salamon1996}.

\begin{theorem}
Let $g_w$ be an exact wedge Sasakian metric on a compact manifold with fibered corners $\cS$ of dimension $4n+3$.  Let $E\to \cS$ be a nicely $\Sp(1)$-equivariant flat Euclidean vector bundle on $\cS$.  Then for $k\le 2n+1$, $u\in \cH^k_{\min}(\cS\setminus\pa \cS;E,g_w)$ is $\Sp(1)$-invariant with $u\equiv 0$ if $k$ is odd and $\Xi^a u=u$ for $a\in\{1,2,3\}$ if $k$ is even.  
\label{bs.22}\end{theorem}
\begin{proof}
By  Proposition~\ref{inv.2}, we can assume  $u\in \cH^{k}_{\min}(Y\setminus\pa Y;E,g_w)$. The $\Sp(1)$-invariance is then a consequence of Lemma~\ref{inv.4}. Given \eqref{van.8} and \eqref{van.9}, we can from that point proceed essentially as in the proof of \cite{Galicki-Salamon1996}.  Let us recall the argument for the benefit of the reader.  

As observed by Galicki and Salamon, it suffices to show that $\Xi^1u=\Xi^2u$, for then the result follows from \eqref{van.9} and symmetry between the indices $1,2,3$.  Now, the proof that  $\Xi^1u=\Xi^2u$ relies on the $\Sp(1)$-invariance of $u$.  Indeed, as in the proof of \cite[Theorem~2.3]{Galicki-Salamon1996}, we may choose $h\in\Sp(1)$ such that 
$h_*\Xi^1=\Xi^2$.  Since both $u$ and $\Xi^1u$ are $\Sp(1)$-invariant, this means that
$$
\begin{aligned}
(\Xi^1u)(X_1,\ldots,X_k)&= h_*(\Xi^1u)(X_1,\ldots,X_k)= \Xi^1u((h^{-1})_*X_1,\ldots,(h^{-1})_*X_k) \\
  &=u(\Xi^1(h^{-1})_*X_1,\ldots, \Xi^1(h^{-1})_*X_k)= u((h^{-1})_*h_*\Xi^1(h^{-1})_*X_1,\ldots, (h^{-1})_*h_* \Xi^1(h^{-1})_*X_k) \\
  &= u((h^{-1})_*(h_*\Xi^1)X_1,\ldots, (h^{-1})_*(h_* \Xi^1)X_k) = h_*u(\Xi^2X_1,\ldots,\Xi^2X_k) \\
  &= u(\Xi^2X_1,\ldots,\Xi^2X_k)= (\Xi^2u)(X_1,\ldots,X_k).
\end{aligned}
$$ 

\end{proof}

\begin{remark}
By a result of Cheeger \cite[Theorem~3.4]{ALMP2012}, when $E$ is a trivial flat Euclidean vector bundle, Theorem~\ref{bs.22} implies that the lower and upper middle perversity intersection cohomology groups associated to the smoothly stratified space $\hat{\cS}$ vanish in degree $2k+1$ for $k\in\{0,\ldots,n\}$.
\label{bs.23}\end{remark}

Combining Theorems~\ref{van.4} and \ref{bs.22} also yields a vanishing in Hodge cohomology for the K\"ahler manifold $(B\setminus\pa B, g_B)$.
\begin{corollary}
Let $(B\setminus \pa B,g_B)$ be the K\"ahler manifold corresponding to the quotient of $(Y\setminus\pa Y,g_w)$ by the $\bbS^1$-action generated by some fixed choice of Reeb vector field $\xi\in\{\xi^1,\xi^2,\xi^3\}$.  Then for $k\in \{0,1,\ldots,n\}$, 
$$
   \cH^{2k+1}_{\min}(B\setminus\pa B;E_B,g_B)=\{0\}.
$$
\label{van.10}\end{corollary}
\begin{proof}
By Theorems\ref{van.4} and \ref{bs.22}, 
$$
  \cH^{2k+1}_{\min}(B\setminus\pa B;E_B,g_B)\cap \ker L^*=\{0\},
 $$
 so the result follows from the Lefschetz decomposition \eqref{l2i.23c}.
\end{proof}

For the quaternionic-K\"ahler manifold $(Q\setminus\pa Q,g_{w,QK})$ of Theorem~\ref{bs.7}, let us remark that it is also possible to obtain a vanishing theorem, but proceeding quite differently via the Weitzenb\"ock formula of Semmelmann and Weingart \cite{SW2002} (see also \cite{Homma2006}).

\begin{theorem}
Let $g_{w,QK}$ be the quaternionic-K\"ahler exact wedge metric on $Q=Y/G$ of Theorem~\ref{bs.7}.  Let $E\to Q$ be a flat Euclidean vector bundle and let  $\eth_{w,QK}$ be the Hodge-deRham operator associated to $g_{w,QK}$ and $E$. Then, for $0\le k\le n$, 
\begin{equation}
 \langle \psi, \eth^2_{w,QK}\psi\rangle_{L^2_w} \ge 2\langle \psi,\psi\rangle_{L^2_w} \quad \forall \psi\in \Omega^{2k+1}_c(Q\setminus\pa Q;E),
\label{bs.8a}\end{equation}
where $\langle\cdot,\cdot\rangle_{L^2_w}$ is the $L^2$-inner product associated to $g_{w,QK}$ and the bundle metric of $E$ and  $\Omega^q_c(Q\setminus\pa Q;E)$ is the space of compactly supported smooth $E$-valued forms on $Q\setminus \pa Q$.  In particular, for $0\le k\le n$,
$$
   \cH^{2k+1}_{\min}(Q\setminus\pa Q; E,g_{w,QK})=\{0\}.
$$
\label{bs.8}\end{theorem}
\begin{proof}
In \cite{SW2002}, Semmelmann and Weingart give a detailed description of the curvature term $R_{QK}$ in the Weitzenb\"ock formula 
\begin{equation}
    \eth_{w,QK}^2= \nabla^*\nabla +R_{QK}
\label{bs.9}\end{equation}
by decomposing it in terms of the irreducible representations of the holonomy group $\Sp(1)\cdot\Sp(n)$ of $g_{w,QK}$.  No flat Euclidean vector bundle was considered in \cite{SW2002}, but since the formula is local, notice that it also holds for the Hodge-deRham operator acting on $E$-valued forms.   For $0\le k\le n$, they obtain the following estimate on $R_{QK}$ acting on ($E$-valued) forms of degree $2k+1$,
\begin{equation}
R_{QK}\ge \frac{\kappa_{w,QK}}{8n(n+2)},
\label{bs.10}\end{equation}
where $\kappa_{w,QK}$ is the scalar curvature of $g_{w,QK}$.  This estimate is not explicitly written in \cite{SW2002}, but it follows from \cite[Lemma~6.2]{SW2002} combined with \cite[(19)]{SW2002}, \cite[Theorem~6.1]{SW2002} and the way \cite[Theorem~4.4]{SW2002} is used in its proof.  Since $\kappa_{w,QK}=16n(n+2)$ by Theorem~\ref{bs.7}, this means that
\begin{equation}
    R_{QK}\ge 2.
\label{bs.11}\end{equation}
The result is then a direct consequence of \eqref{bs.9} and \eqref{bs.11}.

\end{proof}

\section{Reduced $L^2$-cohomology of quiver varieties}\label{rcqv.0}

Together with Corollary~\ref{cqac.21}, the vanishing result of Theorem~\ref{bs.22} will allow us to use the pseudodifferential calculus of \cite{KR1} to prove the Vafa-Witten conjecture.  However, to  be able to proceed by recurrence in the use of Theorem~\ref{bs.22}, we need to be more specific about the type of nicely $\Sp(1)$-equivariant flat Euclidean vector bundles we will consider.
\begin{definition}
Let $g_w$ be an exact $3$-Sasakian metric on a compact manifold with fibered corners $\cS$.  Then a \textbf{fully nicely $\Sp(1)$-equivariant flat Euclidean vector bundle } $E\to \cS$ is a $\Sp(1)$-equivariant flat Euclidean vector bundle on $\cS$ such that for $H\in \cM_1(\cS)$ and $s\in S_H$, the restriction of $E$ to the fiber $\phi_H^{-1}(s)$ is also $\Sp(1)$-equivariant with respect to the $\Sp(1)$-action associated to the exact wedge $3$-Sasakian metric on $\phi_H^{-1}(s)$ induced by $g_w$. 
\label{fe.1}\end{definition}
To explain how such vector bundles arise, let $\mathfrak{M}_{\zeta}$ be a (possibly reduced) quiver variety as in Corollary~\ref{cqac.21} with Nakajima metric $g_{\QAC}$.  The associated manifold with fibered corners is then 
$$
\widetilde{\mathfrak{M}}_{\zeta}:= \widetilde{\mu^{-1}(-\zeta)}/G.
$$
By \eqref{pt.1}, for $H\in \cM_1(\widetilde{\mathfrak{M}}_{\zeta})$ with fiber bundle $\phi_H: H\to S_H$, the metric is asymptotically modelled on a metric of the form
$$
    \frac{du_H^2}{u_H^4} + \pr_1^{*} \frac{\phi_H^*g_{S_H}}{u_H^2} + \pr_1^* \kappa_H
$$
with $g_{S_H}$ an exact wedge metric on $S_H$ and $\kappa_H$ a family of fiberwise $\QAC$-metrics in the fibers of $\phi_H: H\to S_H$ seen as  a $2$-tensor on $H$ with respect to some connection for the bundle $\phi_H:H\to S_H$ .  By the proof of Theorem~\ref{cqac.14}, if the boundary hypersurface $H$ is associated to the the conjugacy class of the stabilizer $G_m$, then for each $s\in S_H$, $(\phi_H^{-1}(s),\kappa_H|_{\phi_H^{-1}(s)})$ corresponds to the quiver variety $\hat{\mu}_m^{-1}(-\zeta_m)/G_m$ with Nakajima metric $g_m$.  On the other hand, by Example~\ref{ew3S}, the metric $g_{S_H}$ is an exact wedge $3$-Sasakian metric on $S_H$.  

Suppose now that the the quiver variety $\hat{\mu} _m^{-1}(-\zeta_m)/G_m$ has finite dimensional reduced $L^2$-cohomology, so a finite dimensional space of $L^2$-harmonic forms.  In that case, there is a corresponding vector bundle $E_H\to S_H$ of vertical  $L^2$-harmonic form with fiber $E_{H,s}$ above $s\in S_H$ corresponding to the space of $L^2$-harmonic forms of $(\phi_H^{-1}(s),\kappa_H|_{\phi_H^{-1}(s)})$. 
\begin{lemma}
The vector bundle $E_H\to S_H$ is a fully nicely $\Sp(1)$-equivariant flat Euclidean vector bundle on $(S_H,g_{S_H})$.   
\label{fe.2}\end{lemma}
\begin{proof}
Notice first that $E_H$ is naturally a Euclidean vector bundle with bundle metric induced by the family of metrics $\kappa_H$.  Now, the local description of \eqref{cor.6a}, \eqref{cor.6b} and \eqref{cor.6c} of $\widetilde{\mu^{-1}(-\zeta)}$ near $H$ gives, after passing to the quotient by the action of $G$, a local trivialization of the fiber bundle $\phi_H:H\to S_H$ that trivializes at the same time the connection induced by the distribution orthogonal to the fibers of $\phi_H$ with respect to the metric $\phi_H^*g_{S_H}+ \kappa_H$.  In particular, the induced connection on $E_H$ is flat and preserves the bundle metric of $E_H$, showing that $E_H$ is a flat Euclidean vector bundle.   

Now, the $\Sp(1)$-action on $S_H$ is induced from the $\Sp(1)$-action on the associated Nakajima quiver representation space $\bold{M}$.  This action commutes with the action of $G$, so preserves the stratification of $\bold{M}$ induced by the action of $G$.  Thus, in the local trivializations of $\phi_H: H\to S_H$ and $E_H\to S_H$ over some  $\cW\subset S_H$, 
$$
    \phi_H^{-1}(\cW)\cong \cW\times \widetilde{\hat{\mu}_m^{-1}(-\zeta_m)/G_m} \quad \mbox{and}\quad E_H|_{\cW}\cong \cW\times \cH^*(\hat{\mu}_m^{-1}(-\zeta_m)/G_m; g_m)
$$
with the action of $\Sp(1)$ on $S_H$ locally lifted to be trivial on the factors 
$$\widetilde{\hat{\mu}_m^{-1}(-\zeta_m)/G_m} \quad \mbox{and} \quad \cH^*(\hat{\mu}_m^{-1}(-\zeta_m)/G_m; g_m)
$$ 
respectively.  This shows in particular that $E_H$ is indeed nicely $\Sp(1)$-equivariant.  

To verify that $E_H$ is fully nicely $\Sp(1)$-equivariant, we need to check that given $H'\in \cM_1(\widetilde{\mu^{-1}(-\zeta)/G})$  such that $H'<H$, the restriction $E_H|_{\phi_{HH'}^{-1}(s)}$ is nicely $\Sp(1)$-equivariant, where $s\in S_{H'}$ and $\phi_{HH'}: S_{HH'}\to S_{H'}$ is the bundle of Definition~\ref{MWFC} with $S_{HH'}\in \cM_1(S_H)$ the boundary hypersurface of $S_H$ associated to $H'$.  Now, the iterated fibration structure of $\widetilde{\mu^{-1}(-\zeta)/G}$ induces one on $\phi_{H'}^{-1}(s)$ and $H\cap \phi^{-1}_{H'}(s)$ is a boundary hypersurface with fiber bundle
$$
   \phi_H: H\cap\phi^{-1}_{H'}(s)\to \phi^{-1}_{HH'}(s)
$$
induced by $\phi_H$.  Moreover the Nakajima metric on $\mathfrak{M}_{\zeta}$ induces a Nakajima metric on $\phi_{H'}^{-1}(s)$ and the $\Sp(1)$-action on $\phi_{HH'}^{-1}(s)$ is induced by the $\Sp(1)$-action on  $\phi_{H'}^{-1}(s)$.  Therefore, working on $\phi^{-1}_{H'}(s)$, we can check that $E_H|_{\phi^{-1}_{HH'}(s)}$ is nicely $\Sp(1)$-equivariant on $\phi^{-1}_{HH'}(s)$ by using the same argument that was used to show that $E_H$ is nicely $\Sp(1)$-equivariant on $S_H$.  
\end{proof}

This lemma will allow us to use Theorem~\ref{bs.22} and apply an argument by induction on the depth of the quiver variety to extract the following result from \cite{KR1}.

\begin{theorem}
For $\zeta\in \bbR^3\times Z$ properly generic, the (possibly reduced) quiver variety $\mathfrak{M}_{\zeta}$ admits a $\QAC$ metric $\widetilde{g}_{\QAC}$ which is $\QAC$ equivalent to the Nakajima metric and such that its space of $L^2$-harmonic forms is finite dimensional and contained in $v^{\epsilon}L^2\Omega^*(\mathfrak{M}_{\zeta}, \widetilde{g}_{\QAC})$ for some $\epsilon>0$.  In particular, the reduced $L^2$-cohomology of the quiver variety is finite dimensional.  
\label{l2c.1}\end{theorem}
\begin{proof}
By Corollary~\ref{cqac.21}, the result will follow from \cite[Theorem~17.5]{KR1} provided that we can check that \cite[Assumptions~17.1,17.2 and 17.4]{KR1} hold for the hyperK\"ahler metric on $\mathfrak{M}_{\zeta}$.  When the $\QAC$ compactification $\widetilde{\mathfrak{M}_{\zeta}}=\widetilde{\mu^{-1}(\zeta)}/G$ of $\mathfrak{M}_{\zeta}$ is of depth 1, notice that \cite[Assumption~17.1]{KR1} is trivially satisfied.  Proceeding by induction on the depth of $\widetilde{\mathfrak{M}_{\zeta}}$, we can suppose more generally that Theorem~\ref{l2c.1} holds for quiver varieties having a $\QAC$ compactification of lower depth.  For $H\in \cM_1(\widetilde{\mathfrak{M}_{\zeta}})$, this means that the hyperK\"ahler metrics of the fibers of $\phi_H: H\to S_H$ have finite dimensional spaces of reduced $L^2$-cohomology, so finite dimensional spaces of $L^2$-harmonic forms.  By Lemma~\ref{fe.2},  \cite[Assumption~17.1]{KR1} holds in this case and the corresponding bundle $E_H\to S_H$ of $L^2$-harmonic forms is a fully nicely $\Sp(1)$-equivariant flat Euclidean vector bundle.  

To complete the proof and the induction, we need to check that \cite[Assumptions~17.2 and 17.4]{KR1} also hold.  First notice that by the result of Hitchin \cite{Hitchin}, the space of $L^2$-harmonic forms of a quiver variety is trivial except possibly in middle degree.  In particular, all the spaces of $L^2$-harmonic forms occurring in \cite[Assumptions 17.2 and 17.4]{KR1} are trivial outside middle degree.  In this case, one can check that \cite[Assumptions~17.4]{KR1} is implied by \cite[Assumption~17.2]{KR1}, so we only need to check the latter.  This assumption requires that for each $H\in\cM_1(\widetilde{\mathfrak{M}_{\zeta}})$, the fully nicely $\Sp(1)$-equivariant flat Euclidean vector bundles $E_H$ on $S_H$ has trivial spaces of $L^2$-harmonic forms in degree $q$ for 
$$
    \left| q-\frac{\dim S_H}2\right|\le 1
$$ 
with respect to the wedge metric $g_{S_H}$ induced by the Nakajima metric of $\mathfrak{M}_{\zeta}$.  By Example~\ref{ew3S}, this metric is an exact wedge $3$-Sasakian metric, so $\dim S_H$ is always odd and we need to check that the space of harmonic forms is trivial in degree $\frac{\dim S_H \pm 1}2$.  Now, by the symmetry of the Hodge star operator, we only need to check this in degree $\frac{\dim S_H- 1}2$, in which case the result follows from Theorem~\ref{bs.22}.

\end{proof}

Using the results of \cite{KR2}, this yields the following characterization of the reduced $L^2$-cohomology of a quiver variety.  

\begin{theorem}
If $(\mathfrak M_{\zeta},g_N)$ is a (possibly reduced) quiver variety equipped with the Nakajima metric $g_N$ and with $\zeta$ properly generic, then 
$$
     \Im[ H^*_c(\mathfrak M_{\zeta})\to H^*(\mathfrak M_{\zeta})] = \cH^*(\mathfrak M_{\zeta}).
$$
\label{l2c.2}\end{theorem}
\begin{proof}
Since smooth quiver varieties are diffeomorphic to smooth affine complex varieties, they have no cohomology above middle degree by a result of Lefschetz \cite[Theorem~7.2]{MilnorMorse}.  Using this property and Corollary~\ref{cqac.21}, the proof of \cite[Corollary~3.3]{KR2} generalizes automatically to quiver varieties with properly generic $\zeta$. Combined with Theorem~\ref{l2c.1}, this allows to generalize the proof of \cite[Theorem~3.5]{KR2} to any quiver variety with properly generic $\zeta$, which yields the result.

\end{proof}

\bibliography{GIQV}

\def\cprime{$'$}
\providecommand{\bysame}{\leavevmode\hbox to3em{\hrulefill}\thinspace}
\providecommand{\MR}{\relax\ifhmode\unskip\space\fi MR }
\providecommand{\MRhref}[2]{%
  \href{http://www.ams.org/mathscinet-getitem?mr=#1}{#2}
}
\providecommand{\href}[2]{#2}
\begin{thebibliography}{10}

\bibitem{ALMP2012}
Pierre Albin, \'{E}ric Leichtnam, Rafe Mazzeo, and Paolo Piazza, \emph{The
  signature package on {W}itt spaces}, Ann. Sci. \'{E}c. Norm. Sup\'{e}r. (4)
  \textbf{45} (2012), no.~2, 241--310. \MR{2977620}

\bibitem{AM2011}
Pierre Albin and Richard Melrose, \emph{Resolution of smooth group actions},
  Spectral theory and geometric analysis, Contemp. Math., vol. 535, Amer. Math.
  Soc., Providence, RI, 2011, pp.~1--26. \MR{2560748}

\bibitem{ALN04}
B.~Ammann, R.~Lauter, and V.~Nistor, \emph{On the geometry of {R}iemannian
  manifolds with a {L}ie structure at infinity}, Internat. J. Math. (2004),
  161--193.

\bibitem{Weiss2021}
Gwyn Bellamy, Alastair Craw, Steven Rayan, Travis Schedler, and Hartmut Weiss,
  \emph{All 81 crepant resolutions of a finite quotient singularity are
  hyperpolygon spaces}, J. Algebraic Geom. \textbf{33} (2024), no.~4, 757--793.
  \MR{4781926}

\bibitem{Boyer-Galicki}
Charles~P. Boyer and Krzysztof Galicki, \emph{Sasakian geometry}, Oxford
  Mathematical Monographs, Oxford University Press, Oxford, 2008. \MR{2382957}

\bibitem{BL92}
J.~Br\"uning and M.~Lesch, \emph{Hilbert complexes}, J. Funct. Anal.
  \textbf{108} (1992), no.~1, 88--132. \MR{1174159}

\bibitem{BL93}
\bysame, \emph{K\"ahler-{H}odge theory for conformal complex cones}, Geom.
  Funct. Anal. \textbf{3} (1993), no.~5, 439--473. \MR{1233862}

\bibitem{Bui}
Quang-Tu Bui, \emph{Injectivity radius of manifolds with a {L}ie structure at
  infinity}, Ann. Math. Blaise Pascal \textbf{29} (2022), no.~2, 235--246.
  \MR{4552719}

\bibitem{Carron2011}
G.~Carron, \emph{On the quasi-asymptotically locally {E}uclidean geometry of
  {N}akajima's metric}, J. Inst. Math. Jussieu \textbf{10} (2011), no.~1,
  119--147. \MR{2749573}

\bibitem{Carron2011b}
Gilles Carron, \emph{Cohomologie {$L^2$} des vari\'{e}t\'{e}s {QALE}}, J. Reine
  Angew. Math. \textbf{655} (2011), 1--59. \MR{2806104}

\bibitem{CDR}
Ronan Conlon, Anda Degeratu, and Fr\'{e}d\'{e}ric Rochon,
  \emph{Quasi-asymptotically conical {C}alabi-{Y}au manifolds}, Geom. Topol.
  \textbf{23} (2019), no.~1, 29--100, With an appendix by Conlon, Rochon and
  Lars Sektnan. \MR{3921316}

\bibitem{Dancer-Swann}
Andrew Dancer and Andrew Swann, \emph{The geometry of singular quaternionic
  {K}\"ahler quotients}, Internat. J. Math. \textbf{8} (1997), no.~5, 595--610.
  \MR{1468352}

\bibitem{DM2018}
Anda Degeratu and Rafe Mazzeo, \emph{Fredholm theory for elliptic operators on
  quasi-asymptotically conical spaces}, Proc. Lond. Math. Soc. (3) \textbf{116}
  (2018), no.~5, 1112--1160. \MR{3805053}

\bibitem{DK2000}
J.~J. Duistermaat and J.~A.~C. Kolk, \emph{Lie groups}, Universitext,
  Springer-Verlag, Berlin, 2000. \MR{1738431}

\bibitem{EKA1990}
A.~El-Kacimi-Alaoui, \emph{Op\'erateurs transversalement elliptiques sur un
  feuilletage riemannien et applications}, Compositio Math. \textbf{73} (1990),
  no.~1, 57--106.

\bibitem{EKA1986}
A.~El-Kacimi-Alaoui and G.~Hector, \emph{D\'ecomposition de {H}odge basique
  pour un feuilletage riemannien}, Ann. Inst. Fourier \textbf{36} (1986),
  no.~3, 207--227.

\bibitem{Farsi2001}
C.~Farsi, \emph{Orbifold spectral theory}, Rocky mountain journal of
  mathematics \textbf{31} (2001), no.~1, 215--235.

\bibitem{Gaffney1951}
M.P. Gaffney, \emph{The harmonic operator for exterior differential forms},
  Proceedings of the National Academy of Sciences \textbf{37} (1951), 48--50.

\bibitem{Galicki-Salamon1996}
Krzysztof Galicki and Simon Salamon, \emph{Betti numbers of {$3$}-{S}asakian
  manifolds}, Geom. Dedicata \textbf{63} (1996), no.~1, 45--68. \MR{1413621}

\bibitem{GM1975}
S.~Gallot and D.~Meyer, \emph{Op\'erateur de courbure et laplacien des formes
  diff\'erentielles d'une vari\'et\'e{} riemannienne}, J. Math. Pures Appl. (9)
  \textbf{54} (1975), no.~3, 259--284. \MR{454884}

\bibitem{Grieser}
D.~Grieser, \emph{Scales, blow-up and quasi-mode construction}, Contemp. Math.,
  AMS \textbf{700} (2017), 207--266, in geometric and computational spectral
  theory.

\bibitem{HHM2004}
T.~Hausel, E.~Hunsicker, and R.~Mazzeo, \emph{Hodge cohomology of gravitational
  instantons}, Duke Math. J. \textbf{122} (2004), no.~3, 485--548.

\bibitem{Hitchin}
Nigel Hitchin, \emph{{$L^2$}-cohomology of hyperk\"{a}hler quotients}, Comm.
  Math. Phys. \textbf{211} (2000), no.~1, 153--165. \MR{1757010}

\bibitem{Homma2006}
Y.~Homma, \emph{Estimating the eigenvalues of quaternionic {K}\"ahler
  manifolds}, Internat. J. Math. \textbf{17} (2006), no.~6, 665--691.

\bibitem{HM05}
Eugenie Hunsicker and Rafe Mazzeo, \emph{Harmonic forms on manifolds with
  edges}, Int. Math. Res. Not. (2005), no.~52, 3229--3272. \MR{2186793}

\bibitem{Huybrechts}
D.~Huybrechts, \emph{Complex geometry. {A}n introduction}, Universitext,
  Springer-Verlag, Berlin, 2005. \MR{2093043 (2005h:32052)}

\bibitem{Joyce}
D.~D. Joyce, \emph{Compact manifolds with special holonomy}, Oxford
  Mathematical Monographs, Oxford University Press, Oxford, 2000. \MR{1787733
  (2001k:53093)}

\bibitem{KR1}
C.~Kottke and F.~Rochon, \emph{Quasi-fibered boundary pseudodifferential
  operators}, arXiv:2103.16650, to appear in {A}st\'erisque.

\bibitem{KR2}
\bysame, \emph{{$L^2$}-cohomology of quasi-fibered boundary metrics}, Invent.
  Math. \textbf{236} (2024), no.~3, 1083--1131. \MR{4743516}

\bibitem{Kronheimer-Nakajima}
Peter~B. Kronheimer and Hiraku Nakajima, \emph{{Yang-Mills instantons on ALE
  gravitational instantons}}, Math. Ann. \textbf{288} (1990), no.~1, 263--307.

\bibitem{Mayrand}
Maxence Mayrand, \emph{Stratification of singular hyperk\"ahler quotients},
  Complex Manifolds \textbf{9} (2022), no.~1, 261--284. \MR{4475684}

\bibitem{Melrose_London}
R.~B. Melrose, \emph{Compactification of {H}ilbert schemes}, Lecture I in the
  ``{L}ondon summer school and workshop: the {S}en conjecture and beyond" at
  the {U}niversity {C}ollege {L}ondon in {J}une 2017
  \texttt{https://ckottke.ncf.edu/senworkshop/abstracts.html}.

\bibitem{MelroseMWC}
\bysame, \emph{Differential analysis on manifolds with corners}, available at
  \texttt{http://www-math.mit.edu/~rbm/book.html}.

\bibitem{Melrose_CIRM}
\bysame, \emph{Planar {H}ilbert scheme and {L}2 cohomology}, talk at the
  conference ``{A}nalysis, {G}eometry and {T}opology of {S}tratified {S}paces''
  at the {CIRM} in {J}une 2016,
  \texttt{https://www.cirm-math.fr/ProgWeebly/Renc1422/Melrose.pdf}.

\bibitem{Melrose1992}
\bysame, \emph{Calculus of conormal distributions on manifolds with corners},
  Int. Math. Res. Not. (1992), no.~3, 51--61.

\bibitem{MelroseGST}
\bysame, \emph{Geometric scattering theory}, Cambridge University Press,
  Cambridge, 1995.

\bibitem{MilnorMorse}
J.~Milnor, \emph{Morse theory}, Annals of Mathematics Studies, vol. No. 51,
  Princeton University Press, Princeton, NJ, 1963, Based on lecture notes by M.
  Spivak and R. Wells. \MR{163331}

\bibitem{Nakajima1994}
Hiraku Nakajima, \emph{{Instantons on ALE spaces, quiver varieties, and
  Kac-Moody algebras}}, Duke Math. J. \textbf{76} (1994), no.~2, 365--416.

\bibitem{Nakajima_1998}
\bysame, \emph{Quiver varieties and {K}ac-{M}oody algebras}, Duke Math. J.
  \textbf{91} (1998), no.~3, 515--560. \MR{1604167}

\bibitem{Segal-Selby}
Graeme Segal and Alex Selby, \emph{The cohomology of the space of magnetic
  monopoles}, Comm. Math. Phys. \textbf{177} (1996), no.~3, 775--787.
  \MR{1385085}

\bibitem{SW2002}
Uwe Semmelmann and Gregor Weingart, \emph{Vanishing theorems for quaternionic
  {K}\"{a}hler manifolds}, J. Reine Angew. Math. \textbf{544} (2002), 111--132.
  \MR{1887892}

\bibitem{Sen}
A.~Sen, \emph{Dyon-monopole bound states, self-dual harmonic forms on the
  multi-monopole moduli space, and {$\SL(2,\bbZ)$} invariance in string
  theory}, Phys. Lett. B \textbf{329} (1994), 217--221.

\bibitem{Tachibana}
S.~Tachibana, \emph{On harmonic tensors in compact sasakian spaces}, T\^ohoku
  Math. J. \textbf{27} (1965), 271--284.

\bibitem{Tanno}
Sh\^{u}kichi Tanno, \emph{On the isometry groups of {S}asakian manifolds}, J.
  Math. Soc. Japan \textbf{22} (1970), 579--590. \MR{271874}

\bibitem{Vafa-Witten}
C.~Vafa and E.~Witten, \emph{A strong coupling test of {S}-duality}, Nuclear
  Phys B \textbf{431} (1994), 3--77.

\end{thebibliography}
\bibliographystyle{amsplain}

\end{document}